\documentclass[12pt, final]{l4dc2023}

\usepackage{wrapfig}

\usepackage{tcolorbox}
\allowdisplaybreaks

\usepackage{amsmath,amsfonts} 

\DeclareMathOperator*{\argmin}{arg\,min}
\usepackage{color}
\usepackage{lipsum} 

\DeclareMathOperator*{\E}{\mathbb{E}}

\usepackage{enumerate}
\usepackage{enumitem}
\usepackage{siunitx}
\usepackage{cleveref}
\usepackage{physics}

\usepackage[open]{bookmark}

\makeatletter
\newcommand{\RN}[1]{\textup{\uppercase\expandafter{\romannumeral#1}}}

\newcommand{\Rmnum}[1]{\expandafter\@slowromancap\romannumeral #1@}
\makeatother

\usepackage{graphicx}

\usepackage{tikz}
\usepackage{subcaption}
\graphicspath{ {/Users/yingyingli/Dropbox (Personal)/Research Projects/switching control/current_overleaf}}

\usepackage{mathtools}

\DeclarePairedDelimiter\ceil{\lceil}{\rceil}
\newcommand{\R}{\mathbb R}

\newcommand{\A}{\mathcal A}
\newcommand{\F}{\mathcal F}

\newcommand{\Pc}{\mathcal P}
\newcommand{\B}{\mathcal B}
\newcommand{\W}{\mathcal W}

\usepackage{dsfont}
\newcommand{\one}{\mathds{1}}
\newcommand{\nbf}{\noindent\textbf}

\newcommand{\bx}{\mathbf{x}}
\newcommand{\des}{\mathrm{des}}

\DeclareMathOperator{\clip}{clip}

\usepackage{algorithm}
\usepackage[noend]{algorithmic}

 \newtheorem{example}{Example}

\title{Online switching control with stability and regret guarantees}
\usepackage{times}

\author{%
	\Name{Yingying Li} \Email{yingli2@caltech.edu}\\
		\Name{James A.\ Preiss} \Email{japreiss@caltech.edu}\\
			\Name{Na Li} \Email{nali@seas.harvard.edu}\\
   \Name{Yiheng Lin} \Email{yihengl@caltech.edu}\\
			\Name{Adam Wierman} \Email{adamw@caltech.edu}\\
	\Name{Jeff Shamma} \Email{jshamma@illinois.edu}\\
}

\usepackage{lineno}
\begin{document}

\maketitle

\vspace{-35pt}

\begin{abstract}%
This paper considers online switching control
with a finite candidate controller pool, an unknown dynamical system, and unknown cost functions. The candidate controllers can be unstabilizing policies. We only require at least one candidate controller to satisfy certain stability properties, but we do not know which one is stabilizing. 
We design an online algorithm that guarantees finite-gain stability throughout the duration of its execution.
We also provide a sublinear policy regret guarantee compared with the optimal stabilizing candidate controller. Lastly, we numerically test our algorithm on  quadrotor planar flights and compare it with a classical switching control algorithm, falsification-based switching,  and a classical multi-armed bandit algorithm, Exp3 with batches.
\end{abstract}

\vspace{-0.2cm}
\section{Introduction}




This paper considers an online switching control problem with a finite pool of candidate controllers $\{\pi_1, \dots, \pi_N\}$, an unknown nonlinear system $x_{t+1}=f(x_t,u_t, w_t )$ with process noises $w_t$, and unknown (time-varying) cost functions $c_t(x_t, u_t)$. Notice that some candidate controllers can be unstabilizing policies.   We only require at least one candidate controller to be stabilizing, but we may not know which one(s) are the stabilizing controllers.\footnote{A switching control problem with at least one stabilizing candidate controller  is sometimes  called a `feasible' problem in the literature \citep{sajjanshetty2018transient,stefanovic2008safe}.} 
We consider a single-trajectory  setting, where the online switching control algorithm \citep[also called a `supervisor'  in the literature,][]{ hespanha2003overcoming} implements a candidate controller at each stage \textit{without} resetting the system state.
Our goal is to design an online algorithm that both stabilizes the system and optimizes the total cost among the candidate controllers.

Online switching control enjoys a long history of research, see e.g., \citep{hespanha2003overcoming,stefanovic2008safe,al2009switching,patil2021unfalsified},  and wide applications, e.g., power systems \citep{meng2016microgrid,dragivcevic2013supervisory}, healthcare \citep{bin2021hysteresis,marchetti2008improved}, autonomous vehicles \citep{1384369}, Internet of Things \citep{zolanvari2019machine}, etc. 
Online switching control is particularly useful in complex scenarios, such as when the problem has non-continuous uncertainties like unknown system orders \citep{liu2017robust} and hybrid systems \citep{garcia2013optimal};   when  multiple control designs are used, for example, comparing model predictive control and PID control \citep{nikoofard2014design};
and when the controller updates are computationally demanding in real-time \citep{zhou1998essentials}.

In the online switching control literature, 
most papers focus on system stabilization, and various approaches have been proposed, e.g.,  estimation-based supervisory control \citep{hespanha2003overcoming}, performance-based falsification \citep{sajjanshetty2018transient,rosa2011stability,al2009switching,stefanovic2008safe},  multi-model adaptive control \citep{shahab2021asymptotic,kuipers2010multiple}, and others.
As for the optimality analysis, most papers either only analyze convergence/asymptotic optimality, e.g., \citep{shahab2021asymptotic,kuipers2010multiple}, or
discuss the optimality  with respect to a cost function designed for stability purposes, instead of a cost function from ``nature'', e.g., \citep{al2009switching,sajjanshetty2018transient,stefanovic2008safe}. Hence, the non-asymptotic optimality on the actual cost $\sum_t c_t(x_t, u_t)$  is largely under-explored for online switching control.

In contrast, there is rich literature in the online learning area that aims to optimize non-asymptotic performance/regret with respect to the actual cost functions \citep{auer2002nonstochastic,arora2012online}. Since online switching control is closely related to online learning, especially multi-armed bandit (MAB) with memory (each candidate controller is an arm, and the current cost depends on the controllers used previously), it is tempting to leverage MAB-with-memory algorithms for online switching control \citep{lin2022online}. However, with unstabilizing candidate controllers, our problem does not satisfy the uniform bounded costs in the MAB literature \citep{auer2002nonstochastic,arora2012online,lin2022online}. Further,   MAB algorithms may cause \textit{unstable}  systems when some candidate controllers are unstabilizing (see, e.g.,  Figure \ref{fig: example exp3 fail}).


Therefore, a natural question arises: \textit{Can we design an online switching control algorithm with both stability and non-asymptotic optimality/regret guarantees on the true cost functions?}

\paragraph{Contributions.} We design an online switching control algorithm Exp3-ISS by integrating an MAB algorithm Exp3 with a stability  certification rule. Exp3-ISS deactivates controllers that fail the stability certification, then switches to other controllers that have not been deactivated.

Theoretically, our Exp3-ISS guarantees finite-gain stability and a sublinear policy regret when compared with the optimal stabilizing candidate controller. We prove a   regret bound that scales as $\tilde O(N^{1/3}T^{2/3})+\exp(O(\mathbb M))$, where $T$ is the horizon length, $N$ is the number of candidate controllers, and $\mathbb M$ is the number of candidate controllers without  the desirable stability properties. Notice that $\tilde O(N^{1/3}T^{2/3})$    is the optimal  regret for MAB with memory \citep{dekel2014bandits}, which suggests  the optimality for our online switching control problem due to its close relation to MAB with memory. The regret $\exp(O(\mathbb M))$ is intuitive  if candidate controllers are black boxes, in which case we must try each candidate controller at least once to determine its performance, and trying $\mathbb M$ unstabilizing controllers consecutively may result in exponentially large states and  regrets. 

Numerically, we test Exp3-ISS on  quadrotor planar flight simulations and compare it with Exp3 and the falsification-based switching algorithm in \citep{al2009switching}.


\paragraph{Related work.} 
\textit{Online switching control}   has been studied under different names, e.g., supervisory control \citep{hespanha2003overcoming}, logic-based switching control \citep{aguiar2007trajectory}, and multi-model adaptive control \citep{kuipers2010multiple}. There are  two major types of switching rules: model-estimation-based rules \citep{hespanha2003overcoming} and performance-based rules that do not estimate models \citep{al2009switching}. This paper belongs to the second type. 

Our stability  certification is inspired by \citet{rosa2011stability} and \citet{al2009switching} but is slightly different because our certification is checked at every stage, while the certification in \citet{rosa2011stability,al2009switching} is only checked  every $\Delta_T$ stages, where $\Delta_T$ is determined by their algorithm. The combination of a stability certification and a performance-optimization algorithm  was also discussed  in \citet{rosa2011stability}, but without optimality guarantees.

There are other stability certificates, e.g., control Lyapunov functions  \citep{brunke2022safe}.



\textit{Online control and online learning.} Online control and its connection with online learning (with memory) have attracted a lot of attention recently \citep{wang2009fast,lin2022online,li2021online,kakade2020information,boffi2021regret}. Most papers consider linear systems, but there is a growing interest in nonlinear systems \citep{kakade2020information,boffi2021regret,lin2022online}. This work is mostly related to \citep{lin2022online,arora2012online,dekel2014bandits}. However, these papers all assume uniform bounded cost functions, which corresponds to all candidate controllers being stabilizing in our case. One major contribution of this paper is to guarantee stability via a novel online control design despite unstabilizing candidate controllers.  

Many online control and learning-based control papers assume to know a stabilizing policy  beforehand \citep{lin2022online,agarwal2019online,fazel2018global,li2021online}, which can be restrictive in certain applications. There is a growing interest on online (learning-based) control without prior knowledge of a stabilizing policy. This paper  contributes to this area since we do not know which candidate controller is stabilizing. 
Besides, our result is related with \citet{chen2021black}, which consider online linear control  and provide a regret bound of  $\tilde O(\text{poly}(d)T^{2/3})+\exp(\text{poly}(d))$, where $d$ is the system dimension. Notice that \citet{chen2021black} only consider linear policies so their regret can depend on the system dimension, while our problem considers black-box controllers without  restrictions or knowledge of controller structures for  nonlinear systems, so our regret bound depends on the number of unstabilizing candidate controllers.  It is an interesting future direction to study how to leverage controller structures in online nonlinear control to generate regret bounds that also depend on the system dimension instead of the number of controllers. 



\textit{Reinforcement learning.} This work is also related to model-free reinforcement learning, especially zeroth-order policy gradient for control, which also updates policies based on observed cost performance \citep{fazel2018global,malik2019derivative,li2021distributed}. The  major difference is that we consider a finite policy pool while policy gradient considers a continuous policy pool. Further,  under proper conditions, policy gradient can guarantee every selected controller updates with  small enough gradient steps to be stabilizing, while our problem allows quick updates of controllers at a cost of  potential encounters with unstabilizing policies.

\nbf{Notations.} 
$\|\cdot\|$ refers to 
 the Euclidean norm.
 
 \vspace{-7pt}
\section{Problem formulation}

This paper focuses on an online supervisory/switching control problem.
We consider an unknown nonlinear dynamical system  $x_{t+1}=f(x_t, u_t, w_t)$ and unknown time-varying cost functions $c_t(x_t, u_t)$,
with state $x_t \in \R^n$, action $ u_t\in \R^m$,
and process noise $w_t \in \R^n$.
We consider a \textit{bandit} setting, i.e., we can only observe the value of $c_t(x_t, u_t)$ after observing $x_t$ and implementing $u_t$ at stage $t$.
The process noise $w_t$ is bounded by a known set $\W=\{w: \|w\|_2\leq w_{\max}\}$ and can be obliviously adversarial, i.e., $w_t$ does not depend on the history states and actions.
We consider a finite pool of candidate controllers
\begin{equation}
	\{i\in \Pc_0=\{1, \dots, N\} : u_t=\pi_i(x_t)\}.
\end{equation}
Some candidate controllers may not stabilize the system, and we do not know which controllers stabilize the system. Further, we treat the candidate controllers as black boxes in this paper and do not assume knowledge of their explicit forms, which is convenient for complex controllers, e.g., when the controllers are represented by neural networks. It is left as future work to consider candidate controllers with known structures.


Our goal is to design an online algorithm $\A$ that selects a candidate controller $I_t\in \Pc_0$ at each stage $t$ in order to  both \textit{stabilize} the system and \textit{optimize} the total cost $J_T(\A)$ defined below.

\[
\textstyle
J_T(\A)= \sum_{t=0}^T c_t(x_t(\A), u_t(\A)), \quad \text{where }  u_t(\A)=\pi_{I_t}(x_t(\A)).
\]

In the supervisory control literature, this online algorithm is often called a ``supervisor'' \citep{hespanha2001tutorial,hespanha2003overcoming,tsao2001unfalsified,al2009switching}.
%
%
%
We now formally introduce our assumptions and our performance metric, policy regret.

\paragraph{1) Assumptions on the candidate controllers.} 
In our problem, we do not need all the candidate controllers to be stabilizing controllers. In fact, we only require at least one of them to satisfy desirable stability properties, which are formally introduced below.

Firstly, we consider input-to-state stability (ISS), which is commonly used in nonlinear systems with process noises $w_t$ \citep{sontag2008input}. Further,  for the purpose of non-asymptotic analysis, we consider exponential-ISS (E-ISS)  below (see e.g., \cite{shi2021meta,8430946}). 
\begin{definition}[E-ISS]\label{def: exp ISS}
	A controller $\pi$ is called exponential-ISS (E-ISS)  with  parameters $(\kappa, \rho,\beta)$ if, for any $x_0\in\R^n$ and $\|w_t\|_2\leq w_{\max}$ for all $t\geq 0$, 
	 the trajectory $x_{t+1}=f(x_t,\pi(x_t), w_t)$ satisfies
 $\|x_t\|_2\leq \kappa \rho^t \|x_0\|_2+ \beta w_{\max}$.\footnote{
 Strictly speaking, this is a relaxed version of E-ISS since we do not require exponentially decaying dependence on history disturbances as in \citep{shi2021meta}.}
\end{definition}

In addition, we consider incremental stability ($\delta$-S), which is commonly adopted to  rigorously quantify the dependence of the current states on the history (see e.g., \citet{angeli2002lyapunov,ruffer2013convergent}). For the purpose of  non-asymptotic analysis, we  consider exponentially decaying dependence, i.e., incremental
exponential stability ($\delta$-ES).
\begin{definition}[$\delta$-ES]\label{def: incremental global exp stable}
	A controller $\pi$ is called incrementally  exponentially stable ($\delta$-ES) with parameters $(\kappa, \rho)$ if we have
	$\|x_t-y_t\|_2\leq \kappa \rho^t \|x_0-y_0\|_2$
	for  two trajectories $x_{t+1}=f(x_t,\pi(x_t), w_t)$ and $y_{t+1}=f(y_t,\pi(y_t), w_t)$ with any $x_0, y_0\in\R^n$ and any $\|w_s\|_2\leq w_{\max}$, $s\leq t-1$.
\end{definition}

\begin{assumption}[On  candidate controllers]\label{ass: one policy good}
	There exists at least one candidate controller $\pi_k$ for $ k\in \Pc_0$ to satisfy Definitions \ref{def: exp ISS}  and \ref{def: incremental global exp stable}  with parameters $(\kappa, \rho,\beta)$, which are known a priori.\footnote{For simplicity, we assume Definition \ref{def: exp ISS} and \ref{def: incremental global exp stable} share the same $\kappa, \rho$, but our results can still hold for different parameters. } 
 Further, 	$\pi_i(x)$ for all $i \in \Pc_0$ are $L_{\pi}$-Lipschitz continuous. We define $\bar \pi_0$ as $\max_{i\in\Pc_0}\|\pi_i(0)\|\leq \bar \pi_{0}$.
\end{assumption}

Notice that there are several important controller designs that satisfy Definitions \ref{def: exp ISS}  and \ref{def: incremental global exp stable}. For example, it is straightforward to verify that stabilizing linear controllers on linear systems satisfy Definitions \ref{def: exp ISS}  and \ref{def: incremental global exp stable}. Similarly, feedback linearization controllers on nonlinear systems also satisfy the two definitions above because the resulting closed-loop system is linear. Furthermore, Definitions \ref{def: exp ISS}  and \ref{def: incremental global exp stable}  can be implied by exponentially incremental ISS (E$\delta$-ISS), which is commonly adopted in the online nonlinear control literature 
\citep{boffi2021regret,tsukamoto2021contraction}. Besides, one can design the controller based on one stability property  and verify the other stability, e.g.,  min-norm policy by an E-ISS control Lyapunov function can also satisfy $\delta$-ES in some cases (see \citep{Yingying2022Supp}).


The candidate controllers can be constructed by e.g., (i) domain knowledge of potentially well-performing policies,  (ii) different control designs  with a finite list of possible policy parameters associated with each control design, (iii)  listing a finite set of possible system dynamics $\mathcal D$ and designing controllers for this set, (iv)  a combination of the methods above, etc. (see e.g., \citep{hespanha2003overcoming} for more discussions). For method (iii), if the true  system belongs to $\mathcal D$ and if the controllers designed for each possible system satisfy the desirable stability properties and the Lipschitz continuity when the corresponding system is the true system, then Assumption \ref{ass: one policy good} is satisfied. In practice, when the true system does not belong to $\mathcal D$ but is close to $\mathcal D$, and if the control design enjoys some robustness, our algorithm can still generate desirable numerical performance as shown in Section \ref{sec:experiments}. Assumption \ref{ass: one policy good} is mostly needed for  theoretical analysis (see  Remarks \ref{remark: assumption for algo implement}-\ref{remark: choose kappa rho beta} in Section \ref{sec: algorithm} for more discussions).



Lastly,  Assumption \ref{ass: one policy good}  assumes to know the  parameters $(\kappa, \beta, \rho)$  a priori, which is for simplicity and was similarly assumed in the online linear control literature \citep{agarwal2019online,minasyan2021online}. Remark \ref{remark: choose kappa rho beta}
briefly discusses how to address the case with unknown parameters.



\paragraph{2) Performance metric.} We measure the optimality  performance of our online algorithm by policy regret, which compares with the optimal policy that satisfies Definitions \ref{def: exp ISS}  and \ref{def: incremental global exp stable}.
\begin{definition}[Policy regret] We define
$\textup{PolicyRegret}(\A)=\E_{(I_t)_{t\geq 0}}J_T(\A)- \min_{i\in \B}J_T(\pi_i),$
where the expectation is over the potentially random controller selection $I_t$ generated by algorithm $\A$ and
	$\B=\{i\in \Pc_0\mid \text{$\pi_i$ satisfies Definitions  \ref{def: exp ISS} and \ref{def: incremental global exp stable}   with parameters $(\kappa, \rho,\beta)$.} \}.$
\end{definition}

In addition, we adopt the finite-gain stability, which is a commonly used stability measure for nonlinear systems with process noise \citep{sastry2013nonlinear}.
\begin{definition}[Finite-gain stability]
For any $1\leq p \leq +\infty$, a system $x_{t+1}=f(x_t, w_t)$ is called finite-gain $l_p$ stable if there exists $0\leq M_1, M_2 <+\infty$  for any $x_0, T$ and any $  w_t\in \mathcal W$ such that 
\[
\textstyle
(\sum_{t=0}^{T}\|x_t\|_2^p)^{1/p} \leq M_1 (\sum_{t=0}^{T}\|w_t\|_2^p)^{1/p}+M_2.
\]

	\end{definition}


\paragraph{3) Assumptions on the dynamics and costs.} We consider Lipschitz continuous nonlinear dynamics with 0 as the equilibrium point below.  

\begin{assumption}[On dynamics]\label{ass: f Lip}
	$f$ is $L_f$-Lipschitz continous with respect to $(x,u,w)$, i.e.,
	for any $x, u, w, x', u', w' \in \R^n$ ($w$ can be in the bounded region), i.e.,
	$|f(x,u,w)-f(x',u', w')|\leq L_f(\|x-x'\|+\|u-u'\|+\|w-w'\|).$
	Further,  $f(0,0,0)=0$.
\end{assumption}
We consider locally Lipschitz continuous cost functions below, which include  quadratic tracking cost $(x_t-\hat x_t)^\top Q (x_t-\hat x_t)+(u_t-\hat u_t)^\top R (u_t-\hat u_t)$ with bounded $\{\hat x_t, \hat u_t\}$  as special cases.
\begin{assumption}[On cost functions]\label{ass: ct}
	There exists $ L_{c1}, L_{c2}$ such that
	$c_t(x,u)$ satisfies the following inequality for any $t$, $x,x'$, $u,u'$:
	$ |c_t(x,u)-c_t(x',u')| \leq (L_{c1}(\max(\|x\|, \|x'\|)+\max(\|u\|,\|u'\|))+L_{c2})(\|x-x'\|+\|u-u'\|) $.
	Further, for all $c_t(x,u)$, there exists $c_0\geq 0$ such that $0\leq c_t(0,0)\leq c_0$.
\end{assumption}
\vspace{-4pt}

For the rest of this paper, we consider {$\kappa\geq 1$, $\beta \geq 1$, $L_f\geq 1$,  $L_{\pi}\geq 1$} for analytical simplicity.\footnote{This is without loss of generality because, if $\kappa<1$ as an example,  we can define $\kappa'=\max(\kappa, 1)$.}

\vspace{-6pt}
\section{Algorithm design}\label{sec: algorithm}
  \begin{figure}
     \centering
     \begin{subfigure}[b]{0.27\textwidth}
         \centering
        \includegraphics[width=\textwidth]{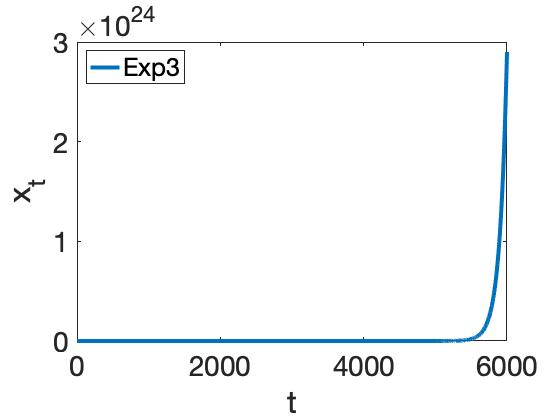}
         \caption{$K_i\in [-1,-0.3, 1]$}
         \label{fig:y equals x}
     \end{subfigure}
     \hfill
     \begin{subfigure}[b]{0.27\textwidth}
         \centering
         \includegraphics[width=\textwidth]{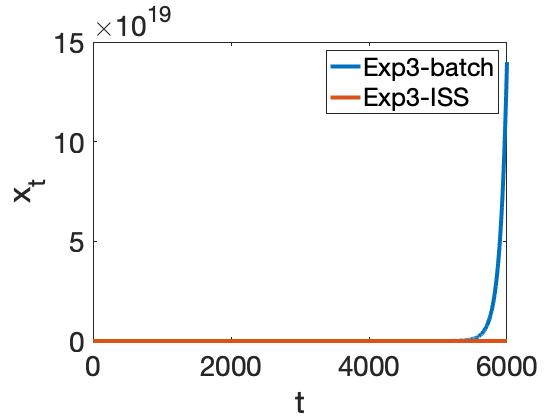}
         \caption{$K_i\in [-1, 0, 1]$}
         \label{fig:three sin x}
     \end{subfigure}
     \hfill
     \begin{subfigure}[b]{0.27\textwidth}
         \centering
         \includegraphics[width=\textwidth]{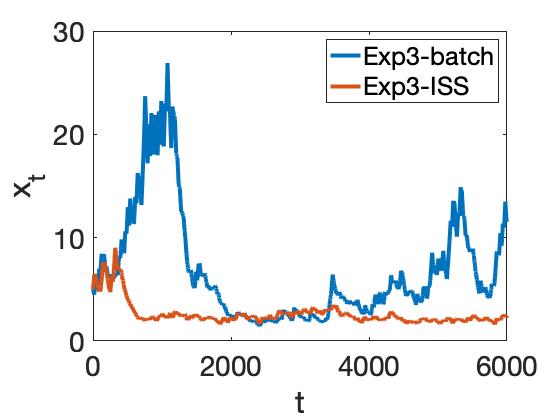}
         \caption{$K_i\in [-1, -0.3, 1]$}
         \label{fig:five over x}
     \end{subfigure}
       	\caption{Examples where Exp3 in \citep{auer2002nonstochastic} and Exp3-batch in \citet{arora2012online,lin2022online} fail to stabilize the system, in comparison to our Exp3-ISS, which stabilizes the system. Consider a   system  $x_{t+1}=x_t +0.01 u_t +w_t$ with $w_t$  i.i.d. generated from $\text{Uniform}[-0.3, 0.7]$. Consider candidate controllers $u_t=K_i x_t$, where $K_i$ are specified in the subfigure captions.}
	\label{fig: example exp3 fail}
\end{figure}


In this section, we introduce our online algorithm for selecting candidate controllers from a  controller pool that may contain unstable controllers.

This problem is closely related to multi-armed bandit (MAB) with memory, by viewing each candidate controller as one arm and noticing that the cost of the current controller depends on the history of the controllers.  Thus, it is tempting to apply MAB (with memory) algorithms to our problem, such as Exp3 \citep{auer2002nonstochastic} and Exp3-batch \citep{lin2022online,arora2012online}.
 However, it is easy to construct examples where Exp3(-batch) fails in this setting.


 \begin{example}[When  Exp3(-batch) fails.]
 \textup{In Figure \ref{fig: example exp3 fail}, we view each controller as an arm and implement  Exp3 \citep{auer2002nonstochastic} and Exp3-batch \citep{lin2022online,arora2012online}, where Exp3-batch is a classical method for MAB with memory.} 
 	
 	\textup{Figure \ref{fig: example exp3 fail}(a) shows that  Exp3 fails to stabilize the system even when a majority of candidate controllers are stabilizing, which is expected due to the memory-dependence of our problem. 
    However, even with batches, Exp3 may still perform poorly, as shown in Figure \ref{fig: example exp3 fail}(b-c). First, when a majority of candidate controllers do not enjoy desirable stability properties (which is exponential stability in this case), Figure \ref{fig: example exp3 fail}(b) shows  Exp3-batch can result in an exponential growth of states. This is because Exp3-batch is only guaranteed to work under  \textit{bounded costs} and \textit{short memory}. However, unstabilizing candidate controllers' costs are \textit{unbounded}, when the unstabilizing candidate controllers already steered the state $x_t$ to be  very large,  the stabilizing controller will also generate a large cost when implemented at stages $t, \dots, t+\tau-1$.
    In other words, the problem  has \textit{long memory} under large states.  Consequently, Exp-batch may fail when there are many unstabilizing candidate controllers. Second, even when the number of unstabilizing candidate controllers is small, Exp3-batch may still perform poorly, as shown in 
  Figure \ref{fig: example exp3 fail}(c), where Exp3-batch generates large spikes in the state trajectory. This is due to explorations of unstabilizing candidates and is not rare because the cost of an unstabilizing candidate controller in one batch may not be forbiddingly large when it starts from a small initial state of this batch thanks to the  stabilizing policies implemented previously.} 
  %
%
\textup{In conclusion, only adding batches to Exp3 is not enough to provide desirable stability performance for online switching control.} 

  
  
  
 	
 	
 	
 	
 	\end{example}


\begin{algorithm}
	\caption{Exp3-ISS}	\label{alg: exp3 iss}
	\begin{algorithmic}[1]
		\STATE  	\textbf{Input:} $(\eta_j)_{j\geq 0}$ where $\eta_j$ is non-increasing. $\tau$. $\kappa, \rho, \beta$. $\tilde G_{-1}(i)=0$ for any $i\in \Pc_0$. A uniform distribution $p_0$ defined on $\Pc_0$. $t_0=0$.
		\FOR{Batch $j=0,1, 2, \dots, $}
		\STATE Initialize $\Pc_{j+1}=\Pc_j$.
		Select $I_j $ from distribution $p_j$. Terminate the algorithm if $\Pc_j $ is empty.
		
		\FOR{$t=t_j, \dots, \min(t_j+\tau-1, T)$}
		\STATE	Implement $\pi_{I_j}$, observe $x_{t+1}$.
		\IF{$ \|x_{t +1}\|_2>\kappa \rho^{t+1-t_j}\|x_{t_j}\|_2 +\beta w_{\max}$}
		\STATE 	Set $\Pc_{j+1}=\Pc_j-\{I_j\}$.
		\STATE \textbf{Break}
		\ENDIF
		\ENDFOR
	\STATE	Let $t_{j+1}=t+1$.
\STATE	Let $g_j(I_j; I_{j-1:0})=\frac{1}{\tau}\sum_{j=t_j}^{t_{j+1}-1} c_t(x_t, u_t)$ and
	 $\tilde g_j(i; I_{j:0})=\frac{g_j(i; I_{j-1:0})}{p_j(i)} \one_{(I_j=i)}$ for  $i\in \Pc_{j+1}$.
\STATE	Let $\tilde G_j(i;I_{j:0})= \tilde G_{j-1}(i;I_{j:0})+ \tilde g_j(i;I_{j:0})$ for  all $i\in \Pc_{j+1}$.
\STATE	Define $$p_{j+1}(i)=\frac{\exp(-\eta_j \tilde G_j(i;I_{j:0}))}{\sum_{k\in \Pc_{j+1}} \exp(-\eta_j\tilde G_j(k;I_{j:0}))}, \quad \forall i\in \Pc_{j+1}.$$
		\ENDFOR
	\end{algorithmic}
\end{algorithm}
%
%
%

\vspace{-0.2cm}

To handle the unstable candidate controllers, we design Exp3-ISS in Algorithm \ref{alg: exp3 iss}. In particular, we utilize Definition \ref{def: exp ISS} to construct an ISS stability certificate (see Line 6 of Algorithm \ref{alg: exp3 iss}). We de-activate the controllers that fail the certificate (Line 6-8),  update the controller selection probabilities $p_{j+1}$ for the active controller pool $\Pc_{j+1}$ (Line 10-12), and  select a controller from the active controller pool at the start of each batch (Line 3). In this way, Exp3-ISS can stabilize the system, which is reflected in Figure \ref{fig: example exp3 fail}(b-c) and will be formally proved in Theorem \ref{thm: finite gain stability}.

\begin{remark}\label{remark: assumption for algo implement}
	 Notice that  Algorithm \ref{alg: exp3 iss} can be implemented as long as there exists an E-ISS candidate controller.  We do not need the controller to also satisfy $\delta$-ES for  implementation and for finite-gain stability in Theorem \ref{thm: finite gain stability}. This can be helpful in practice when $\delta$-ES is difficult to satisfy or verify. 

  Though Assumption \ref{ass: one policy good} requires global stability properties for theoretical analysis, since our Exp3-ISS can guarantee $x_t$ to stay in a relatively small region, local stability properties within this region are already enough for successful implementation of our algorithms. This greatly extends the applicability of our algorithm and is reflected in our numerical experiments in Section \ref{sec:experiments}.
	\end{remark}
\begin{remark}\label{remark: choose kappa rho beta}
If none of the controllers in $\Pc_0$ is E-ISS, Algorithm \ref{alg: exp3 iss} may terminate (Line 3) during implementation since it may de-activate all the controllers. 	If some controllers in $\Pc_0$ are E-ISS, theoretically, we can select large enough $\kappa, \beta$ and $\rho$ close to 1 to ensure at least some controllers can pass the ISS-stability certificate in Line 6 of Algorithm \ref{alg: exp3 iss}, thus avoiding early termination of the algorithm. 	In practice, we can also start with reasonably large $\kappa, \beta, \rho$. If all the controllers are de-activated under the current parameters, we can increase the parameters by, e.g., $\kappa\leftarrow \kappa+\Delta\kappa, \beta\leftarrow \beta+\Delta\beta$ and $\rho\leftarrow \frac{1+\rho}{2}$, then re-start Algorithm \ref{alg: exp3 iss}. If there exists an E-ISS controller in $\Pc_0$, Algorithm \ref{alg: exp3 iss} can still guarantee stability since there will  only be finite times of parameter updates. In practice, if we do not know whether there exists an E-ISS candidate controller, we can adopt additional termination rules, e.g., terminate the algorithm if the updated  $\kappa, \beta, \rho$ exceed certain thresholds.


	
\end{remark}


\vspace{-0.3cm}
\section{Theoretical results}\label{sec: theory}




In this section, we discuss our main results, which provide stability and regret bounds for our online algorithm. 
For ease of reference, we introduce two useful notations below.
First, we define $\mathbb M$  as the number of candidate controllers that do not satisfy Definition \ref{def: exp ISS}.
\begin{definition}\label{def: B0 M0}
	Define $\mathcal B_0$ as the set of controllers that do not satisfy Definition \ref{def: exp ISS} under the $\kappa, \beta, \rho$ used in Algorithm \ref{alg: exp3 iss}. Let
 $\mathbb M$ denote  the number of controllers in $\mathcal B_0$. Notice that $\mathcal B_0 \subseteq \mathcal B^c$.
\end{definition}

Second, we let $J$ denote the number of batches in Algorithm \ref{alg: exp3 iss} for $T$ stages.\footnote{The last batch's index is $J-1$.} It
is shown in our online supplementary material~\citep{Yingying2022Supp}
that $J$ is upper bounded by the following:

\vspace{-0.2cm}
\begin{lemma}[Number of batches]\label{lem: bdd on J} In horizon $T$, the number of batches satisfies
	$J\leq \ceil{\frac{T-\mathbb M}{\tau}}+\mathbb M$.
\end{lemma}


We are now ready to present our stability results.

\vspace{-0.2cm}
\begin{theorem}[Finite-gain stability]\label{thm: finite gain stability}
When  $\tau \geq \frac{\log(2\sqrt 2 \kappa)}{-\log \rho}$,  Algorithm \ref{alg: exp3 iss} is finite-gain $l_1$ stable:

\vspace{0.23cm}
\noindent$\sum_{t=0}^T \|x_t\|\leq  \beta w_{\max} (T+ \alpha_1 J) +\alpha_2 (L_f (1+L_{\pi}) \kappa)^{\mathbb M}\|x_0\|+ \alpha_3 (L_f (1+L_{\pi}) \kappa)^{\mathbb M}(\beta w_{\max}+\bar \pi_0),$
\vspace{0.23cm}

\noindent where $\alpha_1= \frac{\kappa}{1-\rho}\frac{1}{1-\kappa \rho^{\tau}}$, $\alpha_2 = \alpha_1\frac{L_f (1+L_{\pi}) \kappa}{L_f (1+L_{\pi}) \kappa-1}$,  $\alpha_3 = \alpha_2 (\frac{L_f(1+L_{\pi})\kappa}{1-\kappa \rho^{\tau}}+L_f(2+L_{\pi}))$.
Similarly, Algorithm \ref{alg: exp3 iss} also achieves finite gain   $l_2$ stability: 

\vspace{0.23cm}

\noindent$	\sum_{t=0}^T \|x_t\|^2 = O((L_f (1+L_{\pi}) \kappa)^{2\mathbb M}\|x_0\|^2+ \beta^2w_{\max}^2(T+J) +(L_f (1+L_{\pi}) \kappa)^{2\mathbb M}(\beta^2w_{\max}^2+\bar \pi_0^2) )$.
%
\end{theorem}



Theorem \ref{thm: finite gain stability} indicates that Algorithm \ref{alg: exp3 iss} can guarantee bounded states despite unstabilizing controllers in the initial controller pool, which is in contrast with (batch-based) Exp3.

The bound in Theorem \ref{thm: finite gain stability} scales as $O((L_f (1+L_{\pi}) \kappa)^{\mathbb M}+T)$.
The exponential dependence on~$\mathbb M$ can be intuitively explained as follows: since the candidate controllers are black boxes, we must try each controller in $\Pc_0-\mathbb B_0$ at least once to de-activate them.
This may result in exponential growth if we try the controllers in $\Pc_0-\mathbb B_0$ consecutively and these controllers are unstable.
Further, since $\mathbb M$ does not depend on the horizon $T$, the dependence of $\frac{1}{T}\sum_{t=0}^T \|x_t \|$  on $\mathbb M$ will diminish for large enough $T$. 
It is future work to consider non-black-box candidate controllers and leverage the controller structures to reduce the exponential term. 



More specifically, 
when the number of batches $J=o(T)$, and when $T$ goes to infinity, the average $l_1$ norm of the state converges to $\frac{1}{T}\sum_{t=0}^T \|x_t \|\to \beta w_{\max}$. Notice that this is the same state bound achieved by implementing an E-ISS stabilizing controller defined in Definition \ref{def: exp ISS} from the beginning. This suggests that, in the long run, our algorithm can almost recover the performance of the E-ISS stabilizing controllers despite testing unstabilizing controllers at the beginning.


Next, we provide a regret guarantee for our algorithm.
\vspace{-0.2cm}
\begin{theorem}[Policy regret bound]\label{thm: policy regret bdd}
When $\tau \geq \frac{\log(2\sqrt 2 \kappa)}{-\log \rho}$ and $\eta_j=\eta$,  Exp3-ISS's  regret satisfies
	\begin{align*}
		\textup{PolicyRegret}& \leq \alpha_4 \eta NT + (\alpha_5 \eta N\gamma^{4\mathbb M} +\alpha_6  \gamma^{2\mathbb M})\textup{poly}(\|x_0\|,  \bar \pi_0) + \tau\log N/\eta +\alpha_7 J
	\end{align*}
where $\gamma = L_f(1+L_{\pi})\kappa$, $\alpha_4, \dots, \alpha_7$ are  polynomials of $L_f, L_{c1}, L_{c2}, c_0, L_{\pi}, \kappa,\beta w_{\max}, \frac{1}{1-\rho}, \frac{1}{1-2^{3/4}\kappa \rho^{\tau}} $.
\end{theorem}
\vspace{-0.2cm}

\vspace{-0.3cm}
\begin{corollary}[Regret bound order]\label{cor: regret bdd order}
    Let $\eta=O(\frac{1}{N^{2/3} T^{1/3}})$
    and
    $\tau=\max(T^{1/3}N^{-1/3}, \frac{\log(2\sqrt 2 \kappa)}{-\log \rho})$.
    When $T\geq N$,  we have 

    \vspace{-15pt}
    $$\textup{PolicyRegret}\leq   \tilde O( N^{1/3}  T^{2/3})+ \exp(O(\mathbb M)), $$
 where $\tilde O(\cdot)$ hides a $\log(N)$ factor.
\end{corollary}
\vspace{-0.2cm}


Corollary \ref{cor: regret bdd order} shows the order of our regret bound under proper conditions. The first term $ \tilde O(N^{1/3} T^{2/3})$ is common in the policy regret bound of online bandit learning with memory  and has been  shown to be the optimal regret order \citep{dekel2014bandits}. Since online control is closely related to online learning with memory, $  \tilde O( N^{1/3}T^{2/3})$ is likely to also be the optimal regret order for our online control setting. Obtaining a formal lower bound is our ongoing work.

Notice that the exponential term $\exp(O(\mathbb M))$ does not depend on the horizon $T$, so for large enough~$T$, our average regret bound $\text{PolicyRegret}/T$ scales as $ O(1/T^{1/3})$, which diminishes to 0. This indicates that our algorithm can almost recover the optimal  performance of the controllers in~$\mathbb B$ after learning long enough. 
It is also worth mentioning that such an exponential term appears in other online control settings without a stabilization assumption.
For example, in \citep{chen2021black}, the exponential term depends on system dimensionality in a setting with {linear} systems and \textit{linear} controllers,
while our exponential term
depends on the number of unstabilizing controllers since we do not have  knowledge or restrictions on the controller structures. It is our future work to also consider controller structures to improve the exponential term for nonlinear systems.

\vspace{5pt}

\nbf{Proof sketch for Theorem \ref{thm: policy regret bdd}.}
Our proof consists of two parts: we first bound an ``auxiliary regret'' of our algorithm, and then bound the difference between the auxiliary regret and the policy regret. 

\vspace{-0.2cm}
\begin{lemma}[Auxiliary regret bound]\label{lem: pseudo regret}
	Define the \emph{auxiliary regret} of Algorithm \ref{alg: exp3 iss} as
\[
\textstyle
\textup{AuxRegret}(\A)=\tau\E_{(I_j)_{j\geq 0}}\sum_{j=0}^{J-1}g_j(I_j;I_{j-1:0})- \min_{k\in \B}\tau\E_{(I_j)_{j\geq 0}} \sum_{j=0}^{J-1}g_j(k;I_{j-1:0}),
\]
where $I_j$ is seleted by algorithm $\A$. 
		Under the conditions in Theorem \ref{thm: policy regret bdd}, we have	$$\textup{AuxRegret}\leq \alpha_4 \eta N T + \alpha_5  \eta N(L_f(1+L_{\pi})\kappa)^{4\mathbb M} \textup{poly}(\|x_0\|,  \bar \pi_0)  +\tau 
			\log N/\eta. $$
	\end{lemma}

\begin{lemma}[Difference between auxiliary regret and policy regret]\label{lem: difference btw PoR and PsR}
	Under the conditions in Theorem \ref{thm: policy regret bdd}, we have
	$$\textup{PolicyRegret}\leq \textup{AuxRegret} +  \alpha_6 (L_f(1+L_{\pi})\kappa)^{2\mathbb M} \textup{poly}(\|x_0\|, \bar \pi_0) + \alpha_{7} J.$$
	\end{lemma}
\vspace{-0.2cm}

The proof of Theorem \ref{thm: policy regret bdd} follows by combining the bounds in Lemma \ref{lem: pseudo regret} and \ref{lem: difference btw PoR and PsR}.
The detailed proofs of the lemmas are deferred to ~\citep{Yingying2022Supp}. We only discuss some high-level ideas below.
First, auxiliary regret allows the regret benchmark to depend on the same history as that of our algorithm. It is simply called ``regret'' in the classical online learning setting when the cost does not depend on the history decisions. Therefore, we can borrow ideas from the regret bound proof for standard Exp3 to prove Lemma \ref{lem: pseudo regret}. However, standard Exp3 assumes uniformly bounded costs, while our problem suffers unbounded costs. 
To address this issue, we leverage the state bounds in Theorem \ref{thm: finite gain stability}. 
One technical contribution is that we bound the auxiliary regret by the bound on the total cost,   $\sum_j g_j(I_j, I_{j-1:0})$, instead of the uniform bound on $g_j(I_j, I_{j-1:0})$ as in the literature \citep{lin2022online,arora2012online}. This is because the uniform bound on the cost scales as $\exp(O(\mathbb M))$, so directly applying this uniform bound will lead to a regret bound of   order $\exp(O(\mathbb M))T^{2/3}$, which is much worse than our current bound $\tilde O(T^{2/3})+ \exp(O(\mathbb M))$. In fact, the uniform bound is not ideal in our case because we only suffer large states during the transient phase and enjoy small states after unstabilizing controllers are de-activated, which is also reflected in our  numerical results.



Second, Lemma \ref{lem: difference btw PoR and PsR} is the only lemma that utilizes   Definition \ref{def: incremental global exp stable}, which establishes how fast the current state `forgets' the history. When the current state does not depend on the history, the auxiliary regret and the policy regret are identical. Under Definition \ref{def: incremental global exp stable}, the current state forgets the history exponentially fast, so by having a long enough batch size, we can bound the difference between the auxiliary regret and the policy regret. Details are in the supplementary ~\citep{Yingying2022Supp}.

\vspace{-0.2cm}
\section{Numerical experiments}
\label{sec:experiments}
This section provides simulation results on a planar ``quadrotor'' illustrated in Figure \ref{fig:PVTOL}(a) \citep{underactuated}. We consider  state $(x,y,\theta, \dot x, \dot y, \dot \theta)$, where $ (x,y)$ denotes the position and $\theta$ denotes the angle, and  control inputs $(u_1, u_2)$ from the two propellers. The dynamics are
$ m \ddot x = -(u_1+u_2)\sin\theta, \  m \ddot y = -(u_1+u_2)\cos\theta-mg,  \ I \ddot \theta =r(u_1-u_2),$
where $m$ is the mass, $I$ is the moment of inertia, and $r$ is the arm length. Our task is to fly the quadrotor towards a target. 
We consider 81 proportional–derivative candidate controllers as in \citep{lee2010geometric}, whose parameters include  gains $(k_p, k_d, k_p^\theta, k_d^\theta)$  on the position and attitude, and estimations of $m, I, r$. We consider inaccurate estimation of $m$ to 
test the robustness of our algorithm.
More details on the setting are deferred to \citep{Yingying2022Supp} due to space limits.

Figure \ref{fig:PVTOL}(b-c) compare our Exp3-ISS with Exp3-batch in \citep{lin2022online} and Falsification-based Switching (FBS), which  focuses on the stability and does not optimize the cost  \citep{al2009switching}.
When comparing our algorithm with Exp3-batch, we can observe that Exp3-batch performs much worse than our algorithm in terms of both policy regret and the trajectories, with large spikes and fluctuations in the trajectory plot. 
When comparing our algorithm with FBS, we observe that, although FBS performs better than Exp3-ISS at the beginning, FBS generates a linearly increasing regret in expectation, which while Exp3-ISS enjoys regret sublinear in $T$.
This is because FBS ``settles'' on the first stabilizing controller it identifies and does not explore to find better controllers.
Therefore, unless nearly all of the controllers are unstabilizing, FBS
avoids high cost of exploration at the beginning.
However, since FBS essentially selects one stabilizing controller at random, linear regret is unavoidable unless FBS selects the optimal stabilizing controller by random chance.
Figure \ref{fig:PVTOL}(c) shows similar trends: though our algorithm generates larger distances  at the beginning, our distances quickly diminishes to be smaller than FBS after enough exploration.





\begin{figure}
\centering
   
    \begin{subfigure}[t]{0.19\textwidth}
        \raisebox{5mm}{%
        \includegraphics[width=\textwidth]{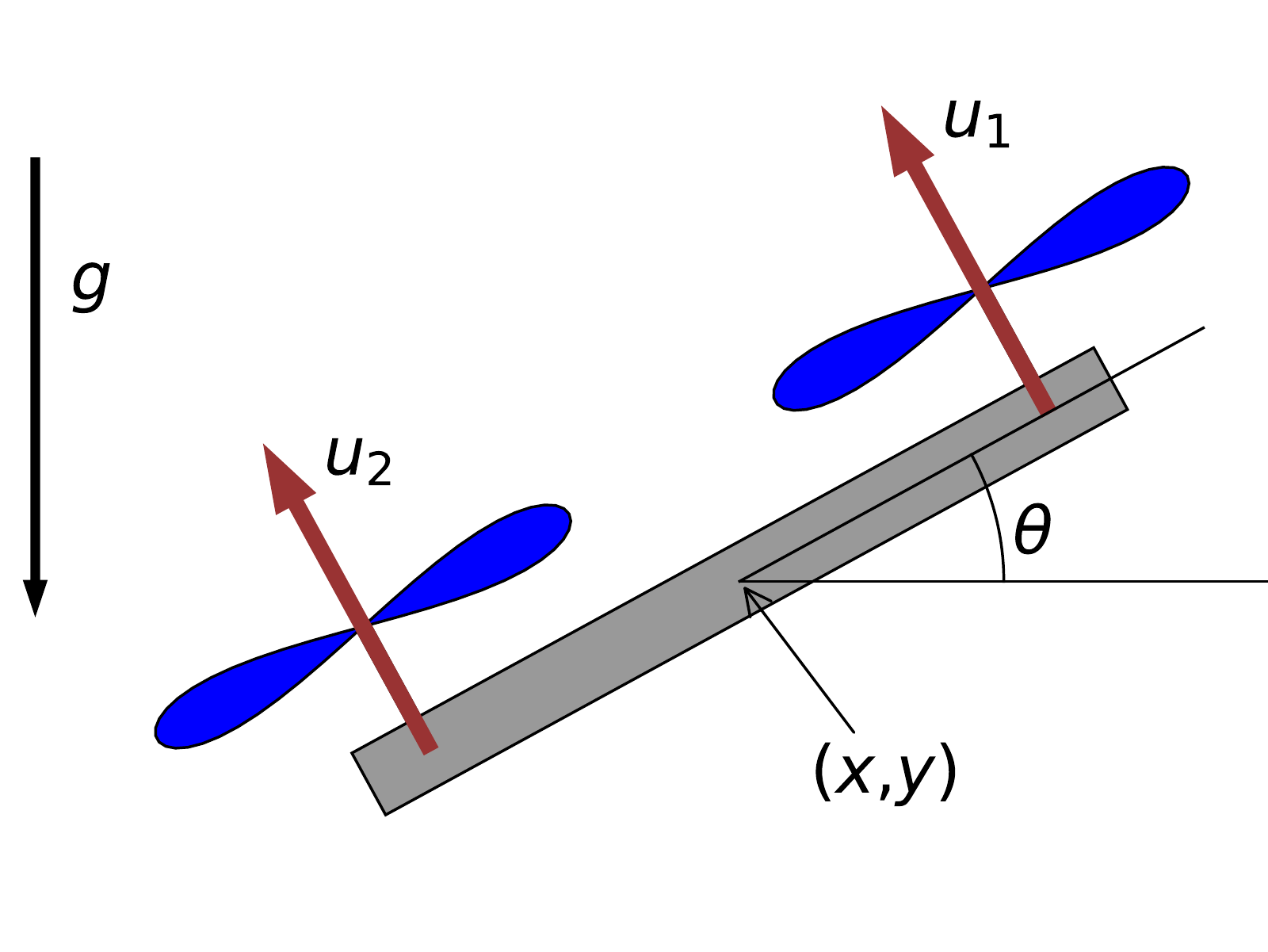}
        }
    	\caption{Planar quadrotor}
    	\label{fig:PVTOL-state}
	\end{subfigure}
    \hfill
    \begin{subfigure}[t]{0.29\textwidth}
     \includegraphics[height=3.2cm, trim={0 0 3.6cm 0}, clip] {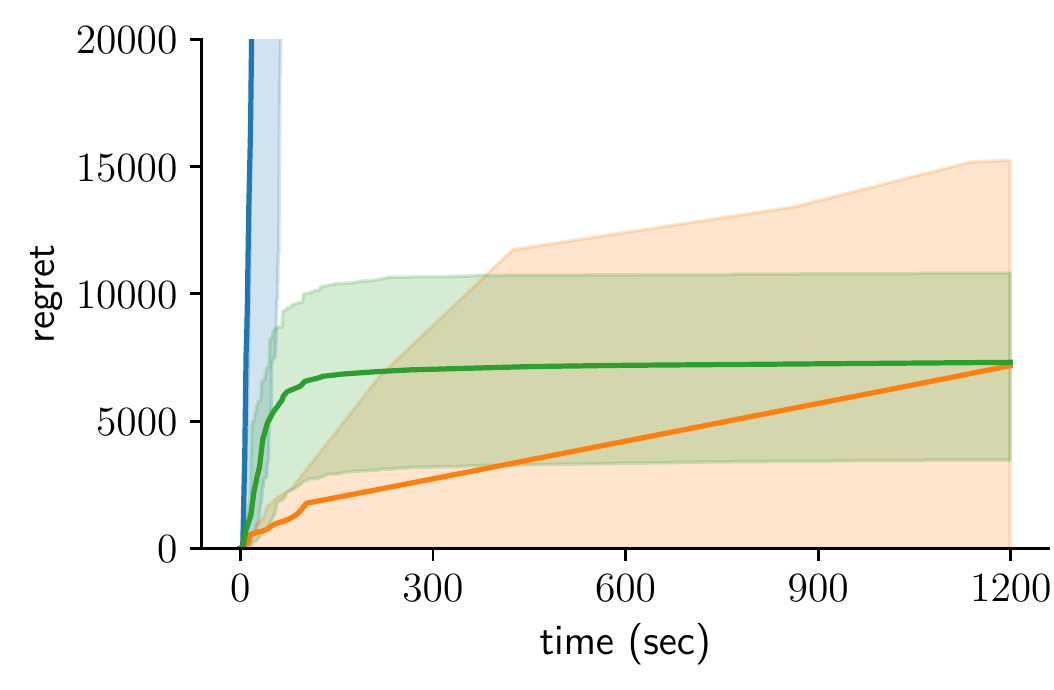}
    	\caption{Policy regret}
    	\label{fig:PVTOL-regret}
	\end{subfigure}
     \hfill
    \begin{subfigure}[t]{0.33\textwidth}
    \includegraphics[height=3.2cm]{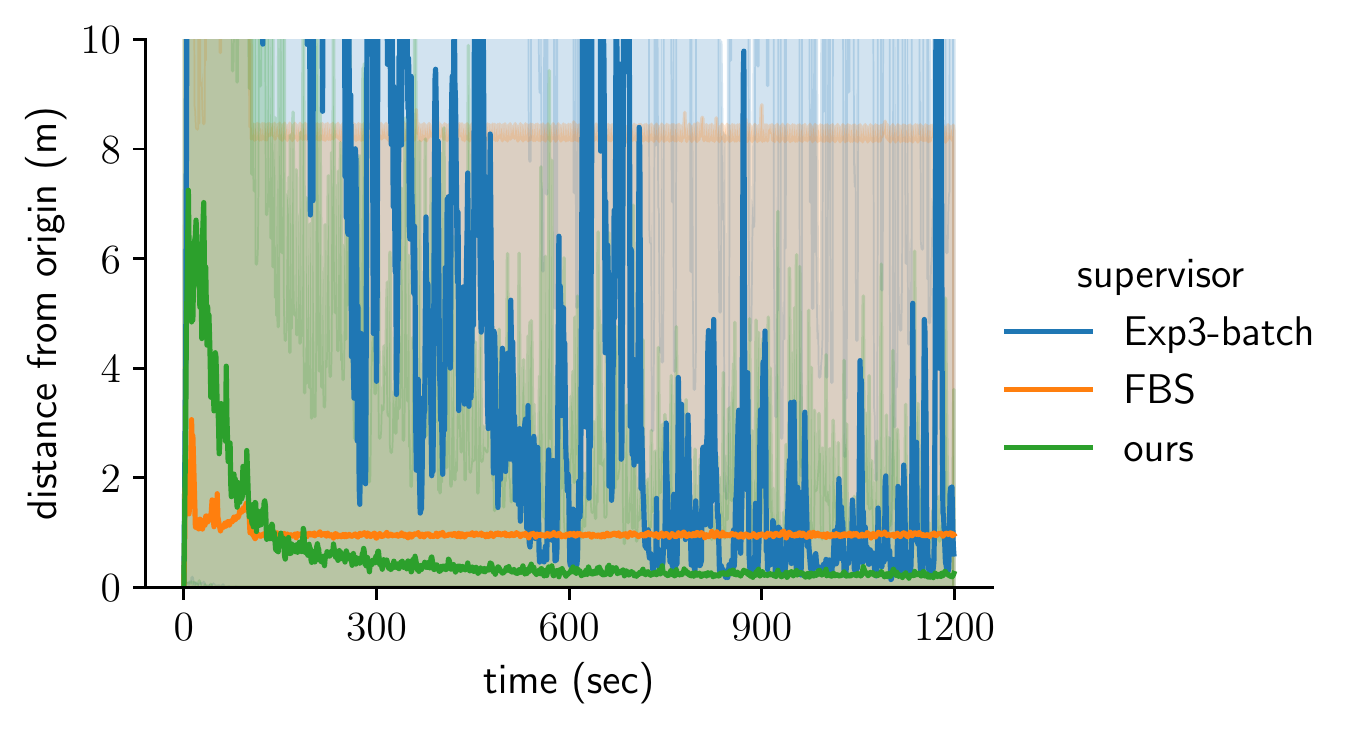}
    	\caption{Distance trajectories}
    	\label{fig:PVTOL-state}
	\end{subfigure}
	\caption{
	    Comparison of \Cref{alg: exp3 iss} with Exp3-batch in \citep{lin2022online} and Falsification-Based Switching (FBS) in \citep{al2009switching} on a simulated planar quadrotor. The solid lines represent the mean value over 100 trials. The shaded regions in (b) and (c) represent the 75\% percentile and the min/max over every trial, respectively.
	}
	\label{fig:PVTOL}
\end{figure}

\vspace{-6pt}
\section{Conclusion and future directions}
This paper proposes an online switching control algorithm by integrating the adversarial bandit algorithm Exp3 with a stability certification. Our algorithm stabilizes the system and provides sublinear policy regret despite the existence of unstabilizing candidate controllers. There are many interesting future directions, e.g., (i) discussing output feedback, where the stability certification in \citep{al2009switching} might be useful, (ii) considering an infinite or continuous policy pool by leveraging problem structure and continuity, (iii) fundamental regret and stability lower bounds for online switching control, (iv) time-varying dynamics where switching policies is necessary for stabilizing the system,
(v) relaxing the global exponential stability assumptions to local and/or asymptotic stability,
and (vi) combining switching-based control with estimation-based control as in multi-model adaptive control, etc.



\bibliography{citation4switching}

\newpage

\appendix

\noindent\textbf{\large Appendices}
\vspace{4pt}

 \paragraph{Notations for the appendices:} 	Denote $t_J=T+1$. Let $\one_S$ denote an indicator function on set $S$, i.e., $\one_S(x)=1$ if and only if $x\in S$.  In addition,  let $1\leq t_{j_1}, \dots, t_{j_M}\leq T$ denote the time indices when Line 6 of Algorithm \ref{alg: exp3 iss}  is activated, i.e., $\|x_{t_{j_s}}\| >\kappa \rho^{t_{j_s}-t_{j_s-1}} \|x_{t_{j_s-1}}\|+\beta w_{\max}$ for $1\leq s \leq M$.  Notice that $1\leq j_1, \dots, j_M\leq J-1$. Also notice that $j_s$ indicates that the previous episode $j_s-1$ terminates by the \textbf{Break} statement. For simplicity, we denote $j_0=0$ and $j_{M+1}=J$, thus,  $t_{j_0}=0$ and  $t_{j_{M+1}}=T+1$.  
 

\section{Proof of Lemma \ref{lem: bdd on J}}
\begin{proof}
	Suppose there are $M$ episodes that terminate when the condition in Line 6 is true in Algorithm \ref{alg: exp3 iss}, then the number of  episodes satisfies
	$J\leq \ceil{\frac{T-M}{\tau}}+M$.

	Notice that the upper bound $\ceil{\frac{T-M}{\tau}}+M$ increases with $M$ when $\tau \geq 1$. Further, notice that $M\leq \mathbb M$ by Definition \ref{def: B0 M0} and Line 6 of Algorithm \ref{alg: exp3 iss}. Therefore, we obtain $J\leq \ceil{\frac{T-\mathbb M}{\tau}}+ \mathbb M$.
\end{proof}

\section{Stability analysis: proof of Theorem \ref{thm: finite gain stability} and supportive lemmas}
In the following, we are going to prove not only Theorem \ref{thm: finite gain stability} but also finite-gain $l_4$ stability, that is,
\begin{equation}\label{equ: l4 stability}
	\sum_{t=0}^T \|x_t\|^4 = O((L_f (1+L_{\pi}) \kappa)^{4\mathbb M}\|x_0\|^4+ \beta^4w_{\max}^4(T+J) +(L_f (1+L_{\pi}) \kappa)^{4\mathbb M}(\beta^4w_{\max}^4+\bar \pi_0^4) ),
	\end{equation}
which will be useful for our regret analysis.

 To prove these finite-gain stability properties, we will first provide a sequence of supportive lemmas on the bounds of the states, which  will also be useful for the regret analysis.

 
 \subsection{Supportive lemmas on the bounds of states in Algorithm \ref{alg: exp3 iss}}
 In this subsection, we provide supportive lemmas on the bounds of the states generated by Algorithm \ref{alg: exp3 iss}. We will discuss the bounds in the $l_2$ norm, $l_2$ norm squared, and $l_2$ norm quartic, which will be used to prove finite-gain $l_1$, $l_2$, and $l_4$ stability, as well as the regret bounds.
 \begin{lemma}[Bounds on states in a single episode]\label{lem: bdd state in one episode by start of episode}
 	In Algorithm \ref{alg: exp3 iss}, at each episode $0\leq j \leq J-1$, for $t_j\leq t \leq t_{j+1}-1$, we have
 	$$\|x_t\|\leq \kappa \rho^{t-t_j}\|x_{t_j}\|+\beta w_{\max}\one_{(t>t_j)}.$$
 	
Consequently, we have
 	\begin{align*}
 		\sum_{t=t_j}^{t_{j+1}-1}\|x_t\|& \leq \frac{\kappa}{1-\rho}\|x_{t_j}\|+\beta w_{\max}(t_{j+1}-t_j-1)\\
 		\sum_{t=t_j}^{t_{j+1}-1}\|x_t\|^2& \leq \frac{2\kappa^2}{1-\rho^2}\|x_{t_j}\|^2+2\beta^2w_{\max}^2(t_{j+1}-t_j-1)\\
 		\sum_{t=t_j}^{t_{j+1}-1}\|x_{t}\|^4& \leq \frac{8\kappa^4}{1-\rho^4}\|x_{t_j}\|^4+8\beta^4w_{\max}^4(t_{j+1}-t_j-1)
 	\end{align*}
 \end{lemma}
 
 \begin{proof}
 	At episode $0\leq j\leq J-1$, no matter whether
 	Algorithm \ref{alg: exp3 iss} breaks at $t_{j+1}$ or not,  for $t_j\leq t \leq t_{j+1}-1$, we have
 	$$\|x_t\|\leq \kappa \rho^{t-t_j}\|x_{t_j}\|+\beta w_{\max} \one_{(t>t_j)},$$
 	which is the first statement of this lemma.
By  H\"{o}lder's inequality, we obtain the following inequalities.
	\begin{align*}
 		\|x_t\|^2 &\leq 2\kappa^2 (\rho^2)^{t-t_j}\|x_{t_j}\|^2 +2\beta^2 w_{\max}^2\one_{(t>t_j)}\\
 		\|x_t\|^4 &\leq 8\kappa^4 (\rho^4)^{t-t_j}\|x_{t_j}\|^4 +8\beta^4 w_{\max}^4\one_{(t>t_j)}
 	\end{align*}
 	Consequently, by summing the three inequalities above over $t=t_j,\dots, t_{j+1}-1$, we  obtain the following.
 	\begin{align*}
 			\sum_{t=t_j}^{t_{j+1}-1}\|x_t\|& \leq \frac{\kappa}{1-\rho}\|x_{t_j}\|+\beta w_{\max}(t_{j+1}-t_j-1)\\
 		\sum_{t=t_j}^{t_{j+1}-1}\|x_t\|^2& \leq \frac{2\kappa^2}{1-\rho^2}\|x_{t_j}\|^2+2\beta^2w_{\max}^2(t_{j+1}-t_j-1)\\
 		\sum_{t=t_j}^{t_{j+1}-1}\|x_t\|^4& \leq \frac{8\kappa^4}{1-\rho^4}\|x_{t_j}\|^4+8\beta^4w_{\max}^4(t_{j+1}-t_j-1)
 	\end{align*}
 \end{proof}

 \begin{lemma}[Relation of states in two consecutive episodes]\label{lem: bdd xt j+1 by xt j}
 	For $0\leq j \leq J-2$,
 	if Algorithm \ref{alg: exp3 iss} activates the \textbf{Break} statement at $x_{t_{j+1}}$, then
 	\begin{align*}
 		\|x_{t_{j+1}}\|&\leq   L_f(1+L_{\pi}) \kappa \rho^{t_{j+1}-t_j-1}\|x_{t_j}\|+ L_f((1+L_{\pi})\beta w_{\max}+  w_{\max}+  \bar \pi_0)\\
 		\|x_{t_{j+1}}\|^2 &\leq 2 L_f^2(1+L_{\pi})^2 \kappa^2 (\rho^2)^{t_{j+1}-t_j-1}\|x_{t_j}\|^2+2 L_f^2((1+L_{\pi})\beta w_{\max}+ w_{\max}+ \bar \pi_0)^2\\
 		\|x_{t_{j+1}}\|^4 &\leq 8 L_f^4(1+L_{\pi})^4 \kappa^4 (\rho^4)^{t_{j+1}-t_j-1}\|x_{t_j}\|^4+8 L_f^4((1+L_{\pi})\beta w_{\max}+ w_{\max}+ \bar \pi_0)^4
 	\end{align*}
 	If Algorithm \ref{alg: exp3 iss} does not activate the \textbf{Break} statement at $x_{t_{j+1}}$, then
 	\begin{align*}
 		\|x_{t_{j+1}}\|&\leq \kappa \rho^{t_{j+1}-t_j}\|x_{t_j}\|+\beta w_{\max}\\
 		\|x_{t_{j+1}}\|^2 &\leq 2\kappa^2 (\rho^2)^{t_{j+1}-t_j}\|x_{t_j}\|^2 +2\beta^2 w_{\max}^2\\
 		\|x_{t_{j+1}}\|^4 &\leq 8\kappa^4 (\rho^4)^{t_{j+1}-t_j}\|x_{t_j}\|^4 +8\beta^4 w_{\max}^4
 	\end{align*}
 \end{lemma}
 \begin{proof}
 	Firstly, we consider the scenario where Algorithm \ref{alg: exp3 iss} activates the \textbf{Break} statement  at $x_{t_{j+1}}$.
 	By Assumption \ref{ass: f Lip}, we have
 	$$\|x_{t+1}\|=	\|f(x_t, u_t, w_t)-f(0,0,0)\|\leq L_f(\|x_t\|+\|u_t\|+\|w_t\|).$$
 Further, by Assumption \ref{ass: one policy good}, we have
 	\begin{align*}
 		\|u_t\|\leq \|u_t-\pi_i(0)\|+\|\pi_i(0)\|\leq L_{\pi}\|x_t\|+\bar \pi_0.
 	\end{align*}
 	Combining the two inequalities above yield the following.
 	\begin{align*}
 		\|x_{t+1}\|\leq L_f(1+L_{\pi})\|x_t\|+ L_f w_{\max}+ L_f \bar \pi_0.
 	\end{align*}
 	Consequently, together with Lemma \ref{lem: bdd state in one episode by start of episode}, we have
 	\begin{align*}
 		\|x_{t_{j+1}}\|&\leq L_f(1+L_{\pi})\|x_{t_{j+1}-1}\|+ L_f w_{\max}+ L_f \bar \pi_0\\
 		& \leq L_f(1+L_{\pi}) \kappa \rho^{t_{j+1}-t_j-1}\|x_{t_j}\|+ L_f(1+L_{\pi})\beta w_{\max}+  L_f w_{\max}+ L_f \bar \pi_0.
 	\end{align*}
 	Therefore, 
 	\begin{align*}
 		\|x_{t_{j+1}}\|^2 &\leq 2 L_f^2(1+L_{\pi})^2 \kappa^2 (\rho^2)^{t_{j+1}-t_j-1}\|x_{t_j}\|^2+2 L_f^2((1+L_{\pi})\beta w_{\max}+ w_{\max}+ \bar \pi_0)^2,\\
 		\|x_{t_{j+1}}\|^4 &\leq 8 L_f^4(1+L_{\pi})^4 \kappa^4 (\rho^4)^{t_{j+1}-t_j-1}\|x_{t_j}\|^4+8 L_f^4((1+L_{\pi})\beta w_{\max}+ w_{\max}+ \bar \pi_0)^4.
 	\end{align*}
 	
 	Secondly, we consider the scenario where Algorithm \ref{alg: exp3 iss} does not activate the \textbf{Break} statement at $x_{t_{j+1}}$.
 	Similarly to Lemma \ref{lem: bdd state in one episode by start of episode}, we have
 	\begin{align*}
 		\|x_{t_{j+1}}\|&\leq \kappa \rho^{t_{j+1}-t_j}\|x_{t_j}\|+\beta w_{\max},\\
 		\|x_{t_{j+1}}\|^2 &\leq 2\kappa^2 (\rho^2)^{t_{j+1}-t_j}\|x_{t_j}\|^2 +2\beta^2 w_{\max}^2,\\
 		\|x_{t_{j+1}}\|^4 &\leq 8\kappa^4 (\rho^4)^{t_{j+1}-t_j}\|x_{t_j}\|^4 +8\beta^4 w_{\max}^4.
 	\end{align*}
 	
 \end{proof}





Now, we are ready to bound the states by discussing the $\textbf{Break}$ activation stages  $ t_{j_1}, \dots, t_{j_M} $, which were initially defined in the \nbf{notations for appendices} at the beginning of the appendices.
 \begin{lemma}[Bounds on states  between two $\textbf{Break}$ activations]\label{lem: bdd xtj btw two breaks}
 	Denote $\gamma_0=2\kappa^2\rho^{2\tau}$. Suppose $2\gamma_0^2<1$, i.e.,   $\tau \geq \frac{\log(2\sqrt 2 \kappa)}{-\log \rho}$.
 	For $1\leq s \leq M+1$, 
 	for $j_{s-1}+1\leq j \leq j_{s}-1$, we have
 	\begin{align*}
 		\|x_{t_j}\|&\leq \left(\sqrt{\gamma_0/2}\right)^{j-j_s-1}\|x_{t_{j_{s-1}}}\|+ \frac{\beta w_{\max}}{1-\sqrt{\gamma_0/2}}\\
 		\|x_{t_j}\|^2& \leq \gamma_0^{j-j_{s-1}}\|x_{t_{j_{s-1}}}\|^2 + \frac{2\beta^2w_{\max}^2}{1-\gamma_0}\\
 		\|x_{t_j}\|^4& \leq (2\gamma_0^2)^{j-j_{s-1}}\|x_{t_{j_{s-1}}}\|^4 + \frac{8\beta^4w_{\max}^4}{1-2\gamma_0^2}
 	\end{align*}
 	Consequently, 
 	\begin{align*}
 		\sum_{j=j_{s-1}}^{j_s-1} \|x_{t_j}\|& \leq  \frac{1}{1-\sqrt{\gamma_0/2}}\|x_{t_{j_{s-1}}}\|+  \frac{\beta w_{\max}}{1-\sqrt{\gamma_0/2}}(j_s-j_{s-1}-1)\\
 		\sum_{j=j_{s-1}}^{j_s-1} \|x_{t_j}\|^2& \leq \frac{1}{1-\gamma_0}\|x_{t_{j_{s-1}}}\|^2 + \frac{2\beta^2w_{\max}^2}{1-\gamma_0}(j_s-j_{s-1}-1)\\
 		\sum_{j=j_{s-1}}^{j_s-1} \|x_{t_j}\|^4& \leq \frac{1}{1-2\gamma_0^2}\|x_{t_{j_{s-1}}}\|^4 + \frac{8\beta^4w_{\max}^4}{1-2\gamma_0^2}(j_s-j_{s-1}-1)
 	\end{align*}
 \end{lemma}
 \begin{proof}
 	For $j_{s-1}+1\leq j \leq j_s-1$, $x_{t_j}$ does not activate \textbf{Break}, hence, we can apply the second scenario in Lemma \ref{lem: bdd xt j+1 by xt j} and obtain
 	\begin{align*}
 		\|x_{t_j}\|\leq \sqrt{\frac{\gamma_0}{2}} \|x_{t_{j-1}}\|+\beta w_{\max}\leq  \left(\sqrt{\frac{\gamma_0}{2}}\right)^{j-j_{s-1}}\|x_{t_{j_{s-1}}}\|+ \frac{\beta w_{\max}}{1-\sqrt{\gamma_0/2}}
 	\end{align*}
 	Similarly, we have
 	$\|x_{t_{j}}\|^2 \leq 2\kappa^2 \rho^{2\tau}\|x_{t_{j-1}}\|^2+2 \beta^2 w_{\max}^2=\gamma_0\|x_{t_{j-1}}\|^2+2 \beta^2 w_{\max}^2 \leq \gamma_0^{j-j_{s-1}}\|x_{t_{j_{s-1}}} \|^2+\frac{2 \beta^2 w_{\max}^2}{1-\gamma_0}$, and $
 	\|x_{t_{j}}\|^4 \leq 8\kappa^4 \rho^{4\tau}\|x_{t_{j-1}}\|^4+8 \beta^4 w_{\max}^4=2\gamma_0^2\|x_{t_{j-1}}\|^4+8 \beta^4 w_{\max}^4 \leq (2\gamma_0^2)^{j-j_{s-1}}\|x_{t_{j_{s-1}}} \|^4+\frac{8 \beta^4 w_{\max}^4}{1-2\gamma_0^2}$.
 	Then, by summing over $j$, we  complete the proof. 
 \end{proof}

 \begin{lemma}[Relation of states at two consecutive $\textbf{Break}$ activations]\label{lem: bdd xtjs by tjs-1}
	Define $\gamma_1=2L_f^2 (1+L_{\pi})^2 \kappa^2$.
	For $1\leq s \leq M$, 
	\begin{align*}
		\|x_{t_{j_s}}\| & \leq \sqrt{\frac{\gamma_1}{2}}  \sqrt{\frac{\gamma_0}{2}}^{j_s-j_{s-1}-1}\|x_{t_{j_{s-1}}}\|+\alpha_8\\
		\|x_{t_{j_s}}\|^2 & \leq \gamma_1\gamma_0^{j_s-j_{s-1}-1}\|x_{t_{j_{s-1}}}\|^2 +\alpha_9\\
		\|x_{t_{j_s}}\|^4 & \leq 2\gamma_1^2(2\gamma_0^2)^{j_s-j_{s-1}-1}\|x_{t_{j_{s-1}}}\|^4 +\alpha_{10}
	\end{align*}
	where $\alpha_8=\sqrt{\frac{\gamma_1}{2}} \frac{\beta w_{\max}}{1-\sqrt{\gamma_0/2}}+ L_f((1+L_{\pi})\beta w_{\max} + w_{\max}+\bar \pi_0)$, $\alpha_9=\frac{2\gamma_1}{1-\gamma_0}\beta^2w_{\max}^2+2L_f^2((1+L_{\pi})\beta w_{\max}+ \bar \pi_0 +w_{\max})^2$, and $\alpha_{10}= 8 L_f^4((1+L_{\pi})\beta w_{\max}+\bar \pi_0+w_{\max})^4+ \frac{16 \gamma_1^2}{1-2\gamma_0^2}\beta^4w_{\max}^4$.
\end{lemma}

\begin{proof}
	By Lemma \ref{lem: bdd xt j+1 by xt j} and Lemma \ref{lem: bdd xtj btw two breaks}, at $t_{j_s}$, we have
	\begin{align*}
		\|x_{t_{j_s}}\|  & \leq L_f (1+L_{\pi})\kappa \rho^{t_{j_s}-t_{j_s-1}}\|x_{t_{j_s-1}}\|+ L_f((1+L_{\pi})\beta w_{\max} + w_{\max}+\bar \pi_0)\\
		& \leq  \sqrt{\frac{\gamma_1}{2}} \|x_{t_{j_s-1}}\|+ L_f((1+L_{\pi})\beta w_{\max} + w_{\max}+\bar \pi_0)\\
		& \leq \sqrt{\frac{\gamma_1}{2}}  \sqrt{\frac{\gamma_0}{2}}^{j_s-j_{s-1}-1}\|x_{t_{j_{s-1}}}\|+ \sqrt{\frac{\gamma_1}{2}} \frac{\beta w_{\max}}{1-\sqrt{\gamma_0/2}}+ L_f((1+L_{\pi})\beta w_{\max} + w_{\max}+\bar \pi_0)
	\end{align*}

	Similarly, we can complete the proof by the following.
	\begin{align*}
		\|x_{t_{j_s}}\|^2  & \leq 2L_f^2(1+L_{\pi})^2\kappa^2 (\rho^2)^{t_{j_s}-t_{j_s-1}}\|x_{t_{j_s-1}}\|^2+2 L_f^2((1+L_{\pi})\beta w_{\max}+ w_{\max}+ \bar \pi_0)^2\\
		& \leq \gamma_1\|x_{t_{j_s-1}}\|^2+2 L_f^2((1+L_{\pi})\beta w_{\max}+ w_{\max}+ \bar \pi_0)^2\\
		& \leq \gamma_1 \gamma_0^{j_s-1-j_{s-1}}\|x_{t_{j_{s-1}}}\|^2 + \gamma_1\frac{2\beta^2w_{\max}^2}{1-\gamma_0} + 2 L_f^2((1+L_{\pi})\beta w_{\max}+ w_{\max}+ \bar \pi_0)^2\\
		\|x_{t_{j_s}}\|^4  & \leq 8 L_f^4(1+L_{\pi})^4 \kappa^4 (\rho^4)^{t_{j+1}-t_j-1}\|x_{t_{j_s-1}}\|^4+8 L_f^4((1+L_{\pi})\beta w_{\max}+ w_{\max}+ \bar \pi_0)^4\\
		& \leq 2\gamma_1^2\|x_{t_{j_s-1}}\|^4+8 L_f^4((1+L_{\pi})\beta w_{\max}+ w_{\max}+ \bar \pi_0)^4\\
		& \leq 2\gamma_1^2 (2\gamma_0^2)^{j_s-1-j_{s-1}}\|x_{t_{j_{s-1}}}\|^4 + 2\gamma_1^2\frac{8\beta^4w_{\max}^4}{1-2\gamma_0^2} + 8 L_f^4((1+L_{\pi})\beta w_{\max}+ w_{\max}+ \bar \pi_0)^4
	\end{align*}
\end{proof}

 Next, we can bound  the starts of  episodes by the following.

 \begin{lemma}[Bounds on the initial states of episodes]\label{lem: bdd on sum xtj}
 	\begin{align*}
 		\sum_{j=0}^{J-1} \|x_{t_j}\| & \leq  \frac{1}{1-\sqrt{\gamma_0/2}}\frac{\sqrt{\frac{\gamma_1}{2}}^{\mathbb M+1}-1}{\sqrt{\frac{\gamma_1}{2}}-1} (\|x_0\|+ \frac{\alpha_8}{\sqrt{\frac{\gamma_1}{2}}-1})+\frac{\beta w_{\max}}{1-\sqrt{\frac{\gamma_0}{2}}} J\\
 		\sum_{j=0}^{J-1} \|x_{t_j}\|^2 & \leq \frac{1}{1-\gamma_0}\frac{\gamma_1^{\mathbb M+1}-1}{\gamma_1-1} (\|x_0\|^2+ \frac{\alpha_9}{\gamma_1-1})+\frac{2\beta^2w_{\max}^2}{1-\gamma_0} J\\
 		\sum_{j=0}^{J-1} \|x_{t_j}\|^4 & \leq \frac{1}{1-2\gamma_0^2}\frac{(2\gamma_1^2)^{\mathbb +1}-1}{2\gamma_1^2-1} (\|x_0\|^4+ \frac{\alpha_{10}}{2\gamma_1^2-1})+ \frac{8\beta^4w_{\max}^4}{1-2\gamma_0^2} J
 	\end{align*}
 \end{lemma}
 
 \begin{proof}
 	Let's first focus on 	$\sum_{j=0}^{J-1} \|x_{t_j}\| $.  By Lemma \ref{lem: bdd xtj btw two breaks}, we have the following.
 	 	\begin{align*}
 		\sum_{j=0}^{J-1}\|x_{t_j}\| & = \sum_{s=0}^M \sum_{j={j_s}}^{j_{s+1}-1}\|x_{t_j}\|\\
 		& \leq \sum_{s=0}^M \left[ \frac{1}{1-\sqrt{\gamma_0/2}}\|x_{t_{j_{s}}}\| +\frac{\beta w_{\max}}{1-\sqrt{\gamma_0/2}}(j_{s+1}-j_{s}-1)\right]\\
 		& =  \frac{1}{1-\sqrt{\gamma_0/2}}\sum_{s=0}^M\|x_{t_{j_{s}}}\|+ \frac{\beta w_{\max}}{1-\sqrt{\gamma_0/2}}J\\
 		& \leq  \frac{1}{1-\sqrt{\gamma_0/2}} \frac{\sqrt{\frac{\gamma_1}{2}}^{\mathbb M+1}-1}{\sqrt{\frac{\gamma_1}{2}}-1} (\|x_0\|+ \frac{\alpha_8}{\sqrt{\frac{\gamma_1}{2}}-1})+ \frac{\beta w_{\max}}{1-\sqrt{\gamma_0/2}}J
 	\end{align*}
 	where the last inequality is because of the following. 
 		For $1\leq s \leq M$, by Lemma \ref{lem: bdd xtjs by tjs-1},  we have
 	$$\|x_{t_{j_s}}\|\leq \sqrt{\frac{\gamma_1}{2}}  \sqrt{\frac{\gamma_0}{2}}^{j_s-j_{s-1}-1}\|x_{t_{j_{s-1}}}\|+\alpha_8\leq \sqrt{\frac{\gamma_1}{2}}  \|x_{t_{j_{s-1}}}\|+\alpha_8,$$
 	where the equalities hold for all $1\leq s \leq M$ when $j_s=s$ for $1\leq s \leq M$, i.e., the first $M$ episodes all activate \textbf{Break}.
 	Consequently, we have 
 	$\|x_{t_{j_s}}\|\leq\sqrt{\frac{\gamma_1}{2}}^s\|x_0\|+ \alpha_8\frac{\sqrt{\frac{\gamma_1}{2}}^{s}}{\sqrt{\frac{\gamma_1}{2}}-1}$
 		and $$	\sum_{s=0}^M \|x_{t_{j_s}}\| \leq \frac{\sqrt{\frac{\gamma_1}{2}}^{M+1}-1}{\sqrt{\frac{\gamma_1}{2}}-1} (\|x_0\|+ \frac{\alpha_8}{\sqrt{\frac{\gamma_1}{2}}-1})\leq  \frac{\sqrt{\frac{\gamma_1}{2}}^{\mathbb M+1}-1}{\sqrt{\frac{\gamma_1}{2}}-1} (\|x_0\|+ \frac{\alpha_8}{\sqrt{\frac{\gamma_1}{2}}-1}),$$
 		where we used $M\leq \mathbb M$ by Definition \ref{def: B0 M0} and by Algorithm \ref{alg: exp3 iss},  and $\gamma_1/2>1$ because we assumed $\kappa\geq 1$ and $L_f\geq 1$ for simplicity.

 	Similarly, by Lemma \ref{lem: bdd xtj btw two breaks} and Lemma \ref{lem: bdd xtjs by tjs-1}, we can complete the proof by the following.
 	\begin{align*}
 		\sum_{j=0}^{J-1} \|x_{t_j}\|^2 & \leq \frac{1}{1-\gamma_0}\sum_{s=0}^M\|x_{t_{j_{s}}}\|^2+ \frac{2\beta^2w_{\max}^2}{1-\gamma_0} J\\
 		&\leq  \frac{1}{1-\gamma_0}\frac{\gamma_1^{\mathbb M+1}-1}{\gamma_1-1} (\|x_0\|^2+ \frac{\alpha_9}{\gamma_1-1})+\frac{2\beta^2w_{\max}^2}{1-\gamma_0}J\\
 		\sum_{j=0}^{J-1} \|x_{t_j}\|^4 & \leq \frac{1}{1-2\gamma_0^2}\sum_{s=0}^M\|x_{t_{j_{s}}}\|^4+ \frac{8\beta^4w_{\max}^4}{1-2\gamma_0^2} J\\
 		&\leq  \frac{1}{1-2\gamma_0^2}\frac{(2\gamma_1^2)^{\mathbb M+1}-1}{2\gamma_1^2-1} (\|x_0\|^4+ \frac{\alpha_{10}}{2\gamma_1^2-1})+ \frac{8\beta^4w_{\max}^4}{1-2\gamma_0^2} J
 	\end{align*}
 \end{proof}
 
 
 \subsection{Proof of Theorem \ref{thm: finite gain stability} and proof of \eqref{equ: l4 stability}}
The proof is straightforward from Lemma \ref{lem: bdd state in one episode by start of episode} and Lemma \ref{lem: bdd on sum xtj}. Let's first consider $\sum_{t=0}^T\|x_t\|$. 
 		\begin{align*}
 		\sum_{t=0}^T\|x_t\|& \leq\sum_{j=0}^{J-1}	\sum_{t=t_j}^{t_{j+1}-1}\|x_t\|\\
 		& \leq \sum_{j=0}^{J-1} \left[\frac{\kappa}{1-\rho}\|x_{t_j}\|+\beta w_{\max}(t_{j+1}-t_j-1)\right]\\
 		&\leq \frac{\kappa}{1-\rho}\sum_{j=0}^{J-1}\|x_{t_j}\|+\beta w_{\max}T\\
 		& \leq \beta w_{\max}T+ \frac{\kappa}{1-\rho}\left(\frac{1}{1-\sqrt{\gamma_0/2}}\frac{\sqrt{\frac{\gamma_1}{2}}^{\mathbb M+1}-1}{\sqrt{\frac{\gamma_1}{2}}-1} (\|x_0\|+ \frac{\alpha_8}{\sqrt{\frac{\gamma_1}{2}}-1})+\frac{\beta w_{\max}}{1-\sqrt{\frac{\gamma_0}{2}}} J\right)\\
 		& =\beta w_{\max} (T+ \alpha_1 J) +\alpha_2 (L_f (1+L_{\pi}) \kappa)^{\mathbb M}\|x_0\|+ \alpha_3 (L_f (1+L_{\pi}) \kappa)^{\mathbb M}(\beta w_{\max}+\bar \pi_0)
 	\end{align*}
 where the last equality is because we defined $\gamma_0=2\kappa^2 \rho^{2\tau}$, $\gamma_1=2L_f^2(1+L_{\pi})^2 \kappa^2$,\\ $\alpha_8\leq \left(L_f(1+L_{\pi}) \kappa/(1-\kappa \rho^{\tau})+ L_f(2+L_{\pi}) \right)(\beta w_{\max}+\bar \pi_0)$, and $\frac{1}{\sqrt{\frac{\gamma_1}{2}}-1}\leq 1$.

 Similarly, we can complete the proof by the following.
 	\begin{align}
 		\sum_{t=0}^T&\|x_t\|^2 \leq\sum_{j=0}^{J-1}	\sum_{t=t_j}^{t_{j+1}-1}\|x_t\|^2 \notag\\
 		& \leq \sum_{j=0}^{J-1} \left[\frac{2\kappa^2}{1-\rho^2}\|x_{t_j}\|^2+2\beta^2w_{\max}^2(t_{j+1}-t_j-1)\right]\notag\\
 		&\leq \frac{2\kappa^2}{1-\rho^2}\sum_{j=0}^{J-1}\|x_{t_j}\|^2+2\beta^2w_{\max}^2T\notag\\
 		& \leq \frac{2\kappa^2}{1-\rho^2} \left(\frac{1}{1-\gamma_0}\frac{\gamma_1^{\mathbb M+1}-1}{\gamma_1-1} (\|x_0\|^2+ \frac{\alpha_9}{\gamma_1-1})+\frac{2\beta^2w_{\max}^2}{1-\gamma_0} J \right) +2\beta^2w_{\max}^2T\notag\\
 		& \leq 2\beta^2w_{\max}^2(T+\alpha_{11} J)+ \alpha_{12} (L_f(1+L_{\pi})\kappa)^{2\mathbb M} \|x_0\|^2 + \alpha_{13} (L_f(1+L_{\pi})\kappa)^{2\mathbb M}(\beta^2 w_{\max}^2+\bar \pi_0^2) \label{equ: l2 stability bdd}\\
 			\sum_{t=0}^T&\|x_t\|^4 \leq \sum_{j=0}^{J-1}	\sum_{t=t_j}^{t_{j+1}-1}\|x_t\|^4\notag\\
 		&  \leq \sum_{j=0}^{J-1} \left[\frac{8\kappa^4}{1-\rho^4}\|x_{t_j}\|^4+8\beta^4w_{\max}^4(t_{j+1}-t_j-1)\right]\notag\\
 		&\leq\frac{8\kappa^4}{1-\rho^4} \sum_{j=0}^{J-1}\|x_{t_j}\|^4+8\beta^4w_{\max}^4 T\notag\\
 		& \leq\frac{8\kappa^4}{1-\rho^4}\frac{1}{1-2\gamma_0^2}\frac{(2\gamma_1^2)^{\mathbb M+1}-1}{2\gamma_1^2-1} (\|x_0\|^4+ \frac{\alpha_{10}}{2\gamma_1^2-1})+\frac{8\kappa^4}{1-\rho^4} \frac{8\beta^4w_{\max}^4}{1-2\gamma_0^2} J+8\beta^4w_{\max}^4T\notag\\
 		& \leq 8\beta^4w_{\max}^4(T+\alpha_{14} J)+ \alpha_{15} (L_f(1+L_{\pi})\kappa)^{4\mathbb M} \|x_0\|^4 + \alpha_{16} (L_f(1+L_{\pi})\kappa)^{4\mathbb M}(\beta^4 w_{\max}^4+\bar \pi_0^4)\label{equ: l4 stability bdd}
 	\end{align}
 where $\alpha_{11}, \dots, \alpha_{16}$ are polynomials of $\kappa, \frac{1}{1-\rho}, \frac{1}{1-2^{3/4}\kappa \rho^{\tau}}, L_f, L_{\pi}$.
 

\section{Regret analysis: proofs of Theorem \ref{thm: policy regret bdd}, Lemma \ref{lem: pseudo regret},  Lemma \ref{lem: difference btw PoR and PsR}, and Corollary \ref{cor: regret bdd order}}
Notice that 	the proof of Theorem \ref{thm: policy regret bdd} is straightforward from combining  Lemma \ref{lem: pseudo regret} and  Lemma \ref{lem: difference btw PoR and PsR}. Further, notice that the proof of Corollary \ref{cor: regret bdd order} is straightforward by plugging in the choices of algorithm parameters in Corollary \ref{cor: regret bdd order} and by using Lemma \ref{lem: bdd on J} to obtain $J\leq O(T/\tau+\mathbb M)$. Therefore, it suffices to prove Lemma \ref{lem: pseudo regret} and Lemma \ref{lem: difference btw PoR and PsR}, which is detailed below.

\subsection{Proof of Lemma \ref{lem: pseudo regret}}
In this proof, we first introduce a supportive lemma, then divide $\text{AuxRegret}$ into separate terms, and provide upper bounds on each terms, which will be combined to prove  Lemma \ref{lem: pseudo regret}.

Firstly, we introduce the supportive lemma below.
\begin{lemma}[Supportive lemma]\label{lem: fj = E tilde fj}
	Conditioning on the natural filtration $\F(I_0, \dots, I_{J-1})$, 
	we have
	$$ g_j(I_j; I_{j-1:0})=\E_{i\sim p_j} \tilde g_j(i; I_{j:0}), \qquad \E_{I_{j:0}} \tilde g_j(k; I_{j:0}) = \E_{I_{j-1:0}}  g_j(k; I_{j-1:0}),$$
for any $k \in \Pc_j$.
\end{lemma}
\begin{proof}
	The first equality is proved by the following.
	\begin{align*}
		\E_{i\sim p_j} \tilde g_j(i; I_{j:0})&= \sum_{i\in \Pc_j} p_j(i)\tilde g_j(i; I_{j:0})\\
		&= \sum_{i\in \Pc_j} p_j(i) \frac{ g_j(i; I_{j-1:0})}{p_j(i)}\one_{(I_j=i)}\\
		&=  g_j(I_j; I_{j-1:0})
	\end{align*}
	
	The second equality is proved by the following.
	\begin{align*}
		\E_{I_{j:0}} \tilde g_j(k; I_{j:0})& = 	\E_{I_{j-1:0}} \E_{I_j}[\tilde g_j(k; I_{j:0})\mid I_{j-1:0}]\\
		& = 	\E_{I_{j-1:0}} \sum_{I_j\in \Pc_j}p_j(I_j)\frac{ g_j(k; I_{j-1:0})}{p_j(k)}\one_{(I_j=k)}\\
		& = 	\E_{I_{j-1:0}}g_j(k; I_{j-1:0})
	\end{align*}
\end{proof}

Secondly, we divide AuxRegret into several terms that are convenient for proving upper bounds. For any $k \in \mathcal B$, we introduce \begin{align}
	\text{AuxRegret}_k=\tau\E_{(I_j)_{j\geq 0}}\sum_{j=0}^{J-1}g_j(I_j;I_{j-1:0})- \tau\E_{(I_j)_{j\geq 0}} \sum_{j=0}^{J-1}g_j(k;I_{j-1:0}).
\end{align}
Notice that it suffices to provide a uniform upper bound on $	\text{AuxRegret}_k$ for all $k\in \mathcal B$ in order to upper bound AuxRegret. Therefore, we will focus on $	\text{AuxRegret}_k$  in the rest of this proof. Notice that we can divide $	\text{AuxRegret}_k$   into several terms below by leveraging Lemma \ref{lem: fj = E tilde fj}.
\begin{align*}
	\text{AuxRegret}_k&=\tau \E_{(I_j)_{j\geq 0}}\sum_{j=0}^{J-1}\E_{i\sim p_j} \tilde g_j(i; I_{j:0})-\tau\E_{(I_j)_{j\geq 0}} \sum_{j=0}^{J-1}g_j(k;I_{j-1:0})
\end{align*}
Further, by adding and subtracting a same term, we have
	\begin{align*}
	\E_{i\sim p_j} \tilde g_j(i; I_{j:0})=& \ \underbrace{ \frac{1}{\eta}\log\left(\E_{i\sim p_j}\exp(-\eta \tilde g_j(i; I_{j:0}) )\right)+ \E_{i\sim p_j} \tilde g_j(i; I_{j-1:0})}_{\text{Term 1}_j}\\
	&\ \underbrace{- \frac{1}{\eta}\log\left(\E_{i\sim p_j}\exp(-\eta \tilde g_j(i; I_{j-1:0}) ))\right)}_{\text{Term 2}_j}.
\end{align*}
Therefore, we can rewrite $\text{AuxRegret}_k$ as the following.
\begin{align}
		\text{AuxRegret}_k&=\tau \E_{(I_j)_{j\geq 0}}\sum_{j=0}^{J-1}\text{Term 1}_j + \tau \E_{(I_j)_{j\geq 0}}\sum_{j=0}^{J-1}\text{Term 2}_j -\tau\E_{(I_j)_{j\geq 0}} \sum_{j=0}^{J-1}g_j(k;I_{j-1:0})
\end{align}

In the following lemmas, we  provide upper bounds on $\E_{(I_j)_{j\geq 0}}\sum_{j=0}^{J-1}\text{Term 1}_j$ and\\ $\E_{(I_j)_{j\geq 0}}\sum_{j=0}^{J-1}\text{Term 2}_j$. The upper bound on $\text{AuxRegret}_k$   is straightforward by combining the upper bounds in Lemma \ref{lem: bdd on E sum Term 3} and Lemma \ref{lem: bdd on sum of Term4}, which completes the proof.
\begin{lemma}[Bounds on the sum of $\text{Term 1}_j$]\label{lem: bdd on E sum Term 3}
	\begin{align*}
		&	\E_{(I_j)_{j\geq 0}} \sum_{j=0}^{J-1}\textup{Term 1}_j \leq \frac{1}{\tau}\left( \alpha_4 \eta NT  \textup{poly}( \beta w_{\max}, \bar \pi_0)+ \alpha_5  \eta N(L_f(1+L_{\pi})\kappa)^{4\mathbb M} \textup{poly}(\|x_0\|, \beta w_{\max}, \bar \pi_0)\right)
	\end{align*}
\end{lemma}

\begin{lemma}[Bounds on the sum of $\text{Term 2}_j$]\label{lem: bdd on sum of Term4}
	\begin{align*}
		\E_{(I_j)_{j\geq 0}} \sum_{j=0}^{J-1} \textup{Term 2}_j \leq \sum_{j=0}^{J-1} \E_{I_{0:j-1}}g_j(k; I_{j-1:0})+\frac{\log N}{\eta}
	\end{align*}
\end{lemma}

The proofs of Lemma \ref{lem: bdd on E sum Term 3} and Lemma \ref{lem: bdd on sum of Term4} are provided in the following.

\subsubsection{Proof of Lemma \ref{lem: bdd on E sum Term 3} }
Firstly, we bound $\text{Term 1}_j$ by the following.
\begin{align}
	\text{Term 1}_j& \leq \frac{1}{\eta}\left(\E_{i\sim p_j}\exp(-\eta \tilde g_j(i; I_{j-1:0}) )-1\right)+ \E_{i\sim p_j} \tilde g_j(i; I_{j-1:0})\notag\\
	& \leq \frac{1}{\eta}\left( \E_{i\sim p_j} \frac{\eta^2}{2} \tilde g_j^2(i; I_{j-1:0}) -\eta  \tilde g_j(i; I_{j-1:0}) )\right)+ \E_{i\sim p_j} \tilde g_j(i; I_{j-1:0})\notag\\
	& =\frac{\eta}{2} \E_{i\sim p_j} \tilde g_j^2(i; I_{j-1:0})\notag \\
	& = \frac{\eta}{2} \sum_{i\in \Pc_{j}} p_j(i) \frac{g_j^2(i; I_{j-1:0})}{p^2_j(i)}\one_{(I_j=i)}\notag\\
	&= \frac{\eta}{2}\frac{g_j^2(I_j; I_{j-1:0})}{p_j(I_j)} \label{equ: bdd on Term 1j}
\end{align}
Secondly, we bound $	\E_{(I_j)_{j\geq 0}}	\text{Term 1}_j$ by the following.
	\begin{align}
	\E_{(I_j)_{j\geq 0}} \text{Term 1}_j &\leq 	\E_{(I_j)_{j\geq 0}} \frac{\eta}{2}\frac{g_j^2(I_j; I_{j-1:0})}{p_j(I_j)}\notag\\
	& = \E_{I_{j-1:0}} \E_{I_j}\left( \frac{\eta}{2}\frac{g_j^2(I_j; I_{j-1:0})}{p_j(I_j)} \mid I_{j-1:0}\right)\notag\\
	& = \E_{I_{j-1:0}}  \sum_{I_j\in \Pc_j}  \left(p_j(I_j) \frac{\eta}{2}\frac{g_j^2(I_j; I_{j-1:0})}{p_j(I_j)} \right)\notag\\
	&= \E_{I_{j-1:0}}  \sum_{I_j\in \Pc_j} \frac{\eta}{2}g_j^2(I_j; I_{j-1:0}) \label{equ: bdd E term 1j}
\end{align}
where the first inequality is by \eqref{equ: bdd on Term 1j}.

Next, we bound $g_j^2(i; I_{j-1:0})$ for any $i \in \Pc_j$. For any $i \in \Pc_j$, we let $\mathring x_{t_j}, \dots, \mathring x_{t_{j+1}-1}$ denote the state trajectory generated by implementing policy $\pi_i$ at episode $j$ and implementing policy $\pi_{I_{j'}}$ at episode $0\leq j'\leq j-1$. Notice that $\mathring x_{t_j}=x_{t_j}$, where $x_{t_j}$ is the state trajectory generated by Algorithm \ref{alg: exp3 iss}. We can bound $g_j^2(i; I_{j-1:0})$  by the following.
\begin{align*}
	g_j^2(i;I_{j-1:0})&=\frac{1}{\tau^2} \left(\sum_{t=t_j}^{t_{j+1}-1} c_t(\mathring x_t,\mathring u_t)\right)^2 \\
	& \leq \frac{1}{\tau^2} \sum_{t=t_j}^{t_{j+1}-1}  c_t^2(\mathring x_t,\mathring u_t)(t_{j+1}-t_j)\\
	& \leq \frac{1}{\tau} \sum_{t=t_j}^{t_{j+1}-1} \left( 2L_{c1}\|\mathring x_t\|^2+2L_{c1}\|\mathring u_t\|^2+L_{c2}\|\mathring x_t\|+L_{c2}\|\mathring u_t\|+c_0 \right)^2\\
	& \leq  \frac{5}{\tau}\sum_{t=t_j}^{t_{j+1}-1} \left( 4L_{c1}^2(\|\mathring x_t\|^4+\|\mathring u_t\|^4)+L_{c2}^2(\|\mathring x_t\|^2+\|\mathring u_t\|^2)+c_0^2 \right)\\
	& = \frac{20L_{c1}^2}{\tau}\sum_{t=t_j}^{t_{j+1}-1}(\|\mathring x_t\|^4+\|\mathring u_t\|^4)+\frac{5L_{c2}^2}{\tau}\sum_{t=t_j}^{t_{j+1}-1}(\|\mathring x_t\|^2+\|\mathring u_t\|^2)+ \frac{5c_0^2}{\tau}(t_{j+1}-t_j),
\end{align*}
where the second inequality is by 
	\begin{align*}
	|c_t(x,u)|& =|c_t(x,u)-c_t(0,0)+c_t(0,0)|\\
	& \leq |c_t(x,u)-c_t(0,0)| + c_0\\
	& \leq L_{c1}(\|x\|+\|u\|)^2+L_{c2}(\|x\|+\|u\|)+c_0\\
	& \leq 2L_{c1}\|x\|^2+2L_{c1}\|u\|^2+L_{c2}\|x\|+L_{c2}\|u\|+c_0.
\end{align*}

By Assumption \ref{ass: one policy good}, we have
$\|\mathring u_t\|^4 \leq 8 L_{\pi}^4\|\mathring x_t\|^4+8 \bar \pi_0^4$ and $\|\mathring u_t\|^2 \leq 2 L_{\pi}^2\|\mathring x_t\|^2+2 \bar \pi_0^2$. Hence, together with Lemma \ref{lem: bdd state in one episode by start of episode}, we have
\begin{align*}
	\sum_{t=t_j}^{t_{j+1}-1}(\|\mathring x_t\|^4+\|\mathring u_t\|^4)& \leq 	\sum_{t=t_j}^{t_{j+1}-1}(\|\mathring x_t\|^4+8 L_{\pi}^4\|\mathring x_t\|^4+8 \bar \pi_0^4)\\
	& = (1+8L_{\pi}^4)\sum_{t=t_j}^{t_{j+1}-1}\|\mathring x_t\|^4+ 8 \bar \pi_0^4(t_{j+1}-t_j)\\
	& \leq  (1+8L_{\pi}^4) \frac{8 \kappa^4}{1-\rho^4} \|x_{t_j}\|^4+ (1+8L_{\pi}^4) 8 \beta^4 w_{\max}^4(t_{j+1}-t_j)+8 \bar \pi_0^4(t_{j+1}-t_j)\\
	\sum_{t=t_j}^{t_{j+1}-1}(\|\mathring x_t\|^2+\|\mathring u_t\|^2)& \leq 	\sum_{t=t_j}^{t_{j+1}-1}(\|\mathring x_t\|^2+2 L_{\pi}^2\|\mathring x_t\|^2+2 \bar \pi_0^2)\\
	& = (1+2L_{\pi}^2)\sum_{t=t_j}^{t_{j+1}-1}\|\mathring x_t\|^2+ 2 \bar \pi_0^2(t_{j+1}-t_j)\\
	& \leq  (1+2L_{\pi}^2) \frac{2 \kappa^2}{1-\rho^2} \| x_{t_j}\|^2+ (1+2L_{\pi}^2) 2 \beta^2 w_{\max}^2(t_{j+1}-t_j)+2 \bar \pi_0^2(t_{j+1}-t_j)
\end{align*}

By combining the inequalities above, we obtain the following uniform upper bound on 	$g_j^2(i;I_{j-1:0})$ for any $i \in \Pc_j$.
\begin{align*}
	g_j^2(i;I_{j-1:0})& \leq \frac{20L_{c1}^2}{\tau} (1+8L_{\pi}^4) \frac{8 \kappa^4}{1-\rho^4} \|x_{t_j}\|^4+ \frac{20L_{c1}^2}{\tau}((1+8L_{\pi}^4) 8 \beta^4 w_{\max}^4+8 \bar \pi_0^4)(t_{j+1}-t_j)\\
	& \quad +\frac{5L_{c2}^2}{\tau}(1+2L_{\pi}^2) \frac{2 \kappa^2}{1-\rho^2} \|x_{t_j}\|^2+ \frac{5L_{c2}^2}{\tau}((1+2L_{\pi}^2) 2 \beta^2 w_{\max}^2+2 \bar \pi_0^2)(t_{j+1}-t_j)\\
	& \quad +\frac{5c_0^2}{\tau}(t_{j+1}-t_j)\\
	& =  \frac{160L_{c1}^2}{\tau} (1+8L_{\pi}^4) \frac{ \kappa^4}{1-\rho^4} \|x_{t_j}\|^4+\frac{10L_{c2}^2}{\tau}(1+2L_{\pi}^2) \frac{ \kappa^2}{1-\rho^2} \|x_{t_j}\|^2\\
	& \quad + \left(  \frac{160L_{c1}^2}{\tau}((1+8L_{\pi}^4)  \beta^4 w_{\max}^4+ \bar \pi_0^4)+ \frac{10L_{c2}^2}{\tau}((1+2L_{\pi}^2)  \beta^2 w_{\max}^2+ \bar \pi_0^2)+\frac{5c_0^2}{\tau}\right)(t_{j+1}-t_j)\\
	& = \frac{1}{\tau}\left(\alpha_{17} \|x_{t_j}\|^4+\alpha_{18} \|x_{t_j}\|^2+\alpha_{19} \text{poly}(\beta w_{\max}, \bar \pi_0) (t_{j+1}-t_j)\right)
\end{align*}
where $\alpha_{17}, \dots, \alpha_{19}$ are polynomials of $\kappa, L_{c1}, L_{c2}, \frac{1}{1-\rho}, L_{\pi}, c_0$.
By plugging the uniform upper bound on $g_j^2(i;I_{j-1:0})$ above to \eqref{equ: bdd E term 1j}, we obtain the following.
	\begin{align}
	\E_{(I_j)_{j\geq 0}} \text{Term 1}_j &\leq \E_{I_{j-1:0}}   \frac{\eta N}{2\tau }\left(\alpha_{17} \|x_{t_j}\|^4+\alpha_{18} \|x_{t_j}\|^2+\alpha_{19} \text{poly}(\beta w_{\max}, \bar \pi_0) (t_{j+1}-t_j)\right)
\end{align}
Therefore, by summing over $j=0,\dots, J-1$, we can prove Lemma \ref{lem: bdd on E sum Term 3} as follows.
\begin{align*}
		\E_{(I_j)_{j\geq 0}} \sum_{j=0}^{J-1}\textup{Term 1}_j & \leq \sum_{j=0}^{J-1}\E_{I_{j-1:0}}   \frac{\eta N}{2\tau }\left(\alpha_{17} \|x_{t_j}\|^4+\alpha_{18} \|x_{t_j}\|^2+\alpha_{19} \text{poly}(\beta w_{\max}, \bar \pi_0) (t_{j+1}-t_j)\right)\\
		& = 	\E_{(I_j)_{j\geq 0}}\sum_{j=0}^{J-1}  \frac{\eta N}{2\tau }\left(\alpha_{17} \|x_{t_j}\|^4+\alpha_{18} \|x_{t_j}\|^2+\alpha_{19} \text{poly}(\beta w_{\max}, \bar \pi_0) (t_{j+1}-t_j)\right)\\
		& \leq \frac{\eta N}{2\tau }  \alpha_{17} \left(\frac{1}{1-\gamma_0}\frac{\gamma_1^{\mathbb M+1}-1}{\gamma_1-1} (\|x_0\|^2+ \frac{\alpha_9}{\gamma_1-1})+\frac{2\beta^2w_{\max}^2}{1-\gamma_0} J\right) \\
		&+ \frac{\eta N}{2\tau } \alpha_{18} \left( \frac{1}{1-2\gamma_0^2}\frac{(2\gamma_1^2)^{\mathbb M +1}-1}{2\gamma_1^2-1} (\|x_0\|^4+ \frac{\alpha_{10}}{2\gamma_1^2-1})+ \frac{8\beta^4w_{\max}^4}{1-2\gamma_0^2} J \right) \\
		&+ \alpha_{19} \text{poly}(\beta w_{\max}, \bar \pi_0) T \frac{\eta N}{2\tau } \\
		&= \frac{1}{\tau}\left( \alpha_4 \eta NT  \text{poly}( \beta w_{\max}, \bar \pi_0)+ \alpha_5  \eta N(L_f(1+L_{\pi})\kappa)^{4\mathbb M} \text{poly}(\|x_0\|, \beta w_{\max}, \bar \pi_0)\right)
\end{align*}
where the last equality is because we defined $\gamma_0=2\kappa^2 \rho^{2\tau}$, $\gamma_1=2L_f^2(1+L_{\pi})^2 \kappa^2$,\\ $\alpha_9=\frac{2\gamma_1}{1-\gamma_0}\beta^2w_{\max}^2+2L_f^2((1+L_{\pi})\beta w_{\max}+ \bar \pi_0 +w_{\max})^2$, and $\alpha_{10}= 8 L_f^4((1+L_{\pi})\beta w_{\max}+\bar \pi_0+w_{\max})^4+ \frac{16 \gamma_1^2}{1-2\gamma_0^2}\beta^4w_{\max}^4$.

\subsubsection{Proof of Lemma \ref{lem: bdd on sum of Term4} }

	\begin{align*}
		\text{Term 2}_j & = - \frac{1}{\eta}\log\left(\E_{i\sim p_j}\exp(-\eta \tilde g_j(i; I_{j-1:0}) )\right)\\
		& = - \frac{1}{\eta}\log\left( \sum_{i\in \Pc_j} p_j(i) \exp(-\eta \tilde g_j(i; I_{j-1:0}) )\right)\\
		& = - \frac{1}{\eta}\log\left( \frac{ \sum_{i\in \Pc_j}  \exp(-\eta \tilde G_{j-1}(i; I_{j-1:0}))\exp(-\eta \tilde g_j(i; I_{j-1:0}) )}{\sum_{i\in \Pc_j}\exp(-\eta \tilde G_{j-1}(i; I_{j-1:0}))  }\right)\\
		& = - \frac{1}{\eta}\log\left( \frac{ \sum_{i\in \Pc_j}  \exp(-\eta \tilde G_{j}(i; I_{j-1:0}))}{\sum_{i\in \Pc_j}\exp(-\eta \tilde G_{j-1}(i; I_{j-1:0}))  }\right)\\
		& = - \frac{1}{\eta}\log\left(\frac{1}{N}\sum_{i\in \Pc_j}  \exp(-\eta \tilde G_{j}(i; I_{j-1:0}))\right)+\frac{1}{\eta}\log\left(\frac{1}{N}\sum_{i\in \Pc_j}\exp(-\eta \tilde G_{j-1}(i; I_{j-1:0})) \right)
	\end{align*}
	where we define $\tilde G_j(i;I_{j-1:0})=\tilde g_j(i;I_{j-1:0})+ \tilde G_{j-1}(i; I_{j-1:0})$ for $i\in\Pc_j- \Pc_{j+1}$ when $j\geq 0$.
	
	For further ease of notations, let's define $\Pc_{-1}=\{1,\dots, N\}$ and\\ $\Phi_j(\eta)= \frac{1}{\eta}\log\left( \frac{1}{N} \sum_{i\in \mathcal P_j} \exp(-\eta \tilde F_j(i;I_{j-1:0}))\right)$ for $j \geq -1$. Notice that $\Phi_{-1}(\eta)=0$ because $\tilde F_{-1}(\cdot)=0$. 
	Then, for $j\geq 0$, we have
	\begin{align*}
		\text{Term 2}_j & =-\Phi_j(\eta)+ \Phi_{j-1}(\eta)+\frac{1}{\eta}\log\left(\frac{\sum_{i\in \Pc_j}\exp(-\eta \tilde G_{j-1}(i; I_{j-1:0}))}{\sum_{i\in \Pc_{j-1}}\exp(-\eta \tilde G_{j-1}(i; I_{j-1:0}))} \right)\\
		& \leq -\Phi_j(\eta)+ \Phi_{j-1}(\eta)
	\end{align*}
	since $\Pc_{j-1}\supseteq \Pc_j$ for $j\geq 0$.
	
	By summing $\text{Term 2}_j$ over $j$, we obtain
	\begin{align*}
		\E_{(I_j)_{j\geq 0}} \sum_{j=0}^{J-1} \text{Term 2}_j& = 	\E_{(I_j)_{j\geq 0}} \sum_{j=0}^{J-1} ( -\Phi_j(\eta)+ \Phi_{j-1}(\eta))\\
		& =	\E_{(I_j)_{j\geq 0}}\Phi_{-1}(\eta)-\Phi_{J-1}(\eta) =	\E_{(I_j)_{j\geq 0}}(-\Phi_{J-1}(\eta) )\\
		& = 	\E_{(I_j)_{j\geq 0}}\frac{-1}{\eta}\log\left( \frac{1}{N} \sum_{i\in \mathcal P_{J-1}} \exp(-\eta \tilde G_{J-1}(i;I_{J-1:0}))\right)\\
		& \overset{(a)}{\leq} 	\E_{(I_j)_{j\geq 0}}\frac{-1}{\eta}\log\left( \frac{1}{N} \exp(-\eta \tilde G_{J-1}(k;I_{J-1:0}))\right)\\
		& = 	\E_{(I_j)_{j\geq 0}}\frac{-1}{\eta}\log\left(  \exp(-\eta \tilde G_{J-1}(k;I_{J-1:0}))\right)+\frac{\log N}{\eta}\\
		& = \E_{(I_j)_{j\geq 0}} \tilde G_{J-1}(k;I_{J-1:0})+\frac{\log N}{\eta}\\
		& = \E_{(I_j)_{j\geq 0}} \sum_{j=0}^{J-1} \tilde g_j(k; I_{j:0})+\frac{\log N}{\eta}\\
		&= \sum_{j=0}^{J-1} \E_{(I_j)_{j\geq 0}} \tilde g_j(k; I_{j:0})+\frac{\log N}{\eta}\\
		& =  \sum_{j=0}^{J-1} \E_{I_{0:j}}\E_{I_{j+1: J-1}} [\tilde g_j(k; I_{j:0})\mid I_{0:j} ]+\frac{\log N}{\eta}\\
		& =  \sum_{j=0}^{J-1} \E_{I_{0:j}}\tilde g_j(k; I_{j:0})+\frac{\log N}{\eta}\\
		& \overset{(b)}{=} \sum_{j=0}^{J-1} \E_{I_{0:j-1}}g_j(k; I_{j-1:0})+\frac{\log N}{\eta}
	\end{align*}
	where $k$ in the inequality (a) is the same $k$ in the definition of $\text{AuxRegret}_k$ and this inequality (a) holds because $k\in \mathcal B\subseteq \Pc_{J-1}$; besides, the equality (b) is because Lemma \ref{lem: fj = E tilde fj} and $k \in \mathcal B\subseteq \Pc_{j}$ for any $0\leq j \leq J-1$.
	This completes the proof.

\subsection{Proof of Lemma  \ref{lem: difference btw PoR and PsR}}
	For ease of notations, we introduce the following definitions. Let $i^*\in \argmin_{i\in \mathcal B} J_T(\pi_i)$. Let $(\hat x_t, \hat u_t)$ for $t\geq 0$ denote the state and action trajectories generated by policy $\pi_{i^*}$, where $\tilde x_0=x_0$. Let  $(x_t, u_t)$ for $t\geq 0$ denote the state and action trajectories generated by our Algorithm \ref{alg: exp3 iss}. Further, for $0\leq j \leq J-1$, define $(\tilde x_t, \tilde u_t)$ for $t_j \leq t\leq t_{j+1}-1$ by the state and action trajectories generated by implementing policy $\pi_{i^*}$ in episode $j$ and implementing the same policies as Algorithm \ref{alg: exp3 iss} in episode $0,\dots, j-1$, where $\tilde x_0=x_0$. Notice that $\tilde x_{t_j}=x_{t_j}$ for all $j$.
	

Then, we have the following relation between the policy regret and the auxiliary regret.
\begin{align*}
\text{PolicyRegret}& =	 \E_{(I_j)_{j\geq 0}} \sum_{t=0}^T c_t(x_t, u_t)- \sum_{t=0}^T c_t(\hat x_t, \hat u_t)\\
& =  \E_{(I_j)_{j\geq 0}}   \sum_{j=0}^{J-1} \sum_{t=t_j}^{t_{j+1}-1}  c_t(x_t, u_t) -  \E_{(I_j)_{j\geq 0}} \sum_{j=0}^{J-1} \sum_{t=t_j}^{t_{j+1}-1} c_t(\tilde x_t, \tilde u_t)\\
& \quad +  \E_{(I_j)_{j\geq 0}} \sum_{j=0}^{J-1} \sum_{t=t_j}^{t_{j+1}-1} c_t(\tilde x_t, \tilde u_t)- \sum_{j=0}^{J-1} \sum_{t=t_j}^{t_{j+1}-1} c_t(\hat x_t, \hat u_t)\\
& =\E_{(I_j)_{j\geq 0}}   \sum_{j=0}^{J-1} \tau  g_j(I_j; I_{j-1:0})- \E_{(I_j)_{j\geq 0}}   \sum_{j=0}^{J-1} \tau  g_j(i^*; I_{j-1:0})\\
& \quad +  \E_{(I_j)_{j\geq 0}} \sum_{j=0}^{J-1} \sum_{t=t_j}^{t_{j+1}-1} c_t(\tilde x_t, \tilde u_t)- \sum_{j=0}^{J-1} \sum_{t=t_j}^{t_{j+1}-1} c_t(\hat x_t, \hat u_t)\\
& \leq \text{AuxRegret}+  \E_{(I_j)_{j\geq 0}} \sum_{j=0}^{J-1} \sum_{t=t_j}^{t_{j+1}-1} c_t(\tilde x_t, \tilde u_t)- \sum_{j=0}^{J-1} \sum_{t=t_j}^{t_{j+1}-1} c_t(\hat x_t, \hat u_t)
\end{align*}
Therefore, it suffices to bound  $\E_{(I_j)_{j\geq 0}} \sum_{j=0}^{J-1} \sum_{t=t_j}^{t_{j+1}-1} c_t(\tilde x_t, \tilde u_t)- \sum_{j=0}^{J-1} \sum_{t=t_j}^{t_{j+1}-1} c_t(\hat x_t, \hat u_t).
$ By Assumption \ref{ass: one policy good}, Assumption \ref{ass: ct}, Definition \ref{def: exp ISS}, and Definition \ref{def: incremental global exp stable}, we have the following.
	\begin{align*}
	&	|c_t(\tilde x_t, \tilde u_t) -c_t(\hat x_t, \hat u_t)| \leq L_{c1}[\max(\|\tilde x_t\|, \|\hat x_t\|)+ \max( \|\tilde u_t\|, \|\hat u_t\|) +L_{c2}](\|\tilde x_t-\hat x_t\|+ \|\tilde u_t-\hat u_t\|) \\
	& \leq  L_{c1}[\kappa \rho^{t-t_j}(1+L_{\pi})\max(\|\tilde x_{t_j}\|, \|\hat x_{t_j}\|)+(1+L_{\pi}) \beta w_{\max}+L_{c2}](1+L_{\pi})\kappa \rho^{t-t_j}\|\tilde x_{t_j}-\hat x_{t_j}\| \\
	& \leq L_{c1}\kappa^2 (\rho^2)^{t-t_j}(1+L_{\pi})^2(\|\tilde x_{t_j}\|+\|\hat x_{t_j}\|)^2\\
	& \quad + L_{c1}[(1+L_{\pi}) \beta w_{\max}+L_{c2}](1+L_{\pi})\kappa \rho^{t-t_j}(\|\tilde x_{t_j}\|+\|\hat x_{t_j}\|)\\
	& \leq L_{c1}\kappa^2 (\rho^2)^{t-t_j}(1+L_{\pi})^2(2\|\tilde x_{t_j}\|^2+2\|\hat x_{t_j}\|^2)\\
	& \quad + L_{c1}[(1+L_{\pi}) \beta w_{\max}+L_{c2}](1+L_{\pi})\kappa \rho^{t-t_j}(\|\tilde x_{t_j}\|+\|\hat x_{t_j}\|)
\end{align*}
Therefore, we have the following bound.
\begin{align*}
&\text{PolicyRegret} \leq \text{AuxRegret}+ \E_{(I_j)_{j\geq 0}} \sum_{j=0}^{J-1} \sum_{t=t_j}^{t_{j+1}-1} (c_t(\tilde x_t, \tilde u_t)-c_t(\hat x_t, \hat u_t))\\
	& \leq \text{AuxRegret}  + \E_{(I_j)_{j\geq 0}} \sum_{j=0}^{J-1}\left[ L_{c1} \frac{\kappa^2}{1-\rho^2}(2\|\tilde x_{t_j}\|^2+ 2\|\hat x_{t_j}\|^2)(L_{\pi}+1)^2 \right. \\
	& \left. + L_{c1} \frac{\kappa}{1-\rho}(\|\tilde x_{t_j}\|+ \|\hat x_{t_j}\|)(L_{\pi}+1)(\beta w_{\max}(L_{\pi}+1)+L_{c2})\right]\\
	& \leq  \text{AuxRegret}  +  L_{c1} \frac{\kappa^2}{1-\rho^2}(L_{\pi}+1)^2(2\E_{(I_j)_{j\geq 0}} \sum_{j=0}^{J-1}\|\tilde x_{t_j}\|^2+ 2\E_{(I_j)_{j\geq 0}} \sum_{j=0}^{J-1}\|\hat x_{t_j}\|^2)\\
	& \quad +  L_{c1} \frac{\kappa}{1-\rho}(L_{\pi}+1)(\beta w_{\max}(L_{\pi}+1)+L_{c2})(\E_{(I_j)_{j\geq 0}} \sum_{j=0}^{J-1}\|\tilde x_{t_j}\|+ \E_{(I_j)_{j\geq 0}} \sum_{j=0}^{J-1}\|\hat x_{t_j}\|)\\
	& \overset{(a)}{\leq}\text{AuxRegret} + 2L_{c1} \frac{\kappa^2}{1-\rho^2}(L_{\pi}+1)^2 \left( \frac{1}{1-\gamma_0}\frac{\gamma_1^{\mathbb M+1}-1}{\gamma_1-1} (\|x_0\|^2+ \frac{\alpha_9}{\gamma_1-1})+\frac{2\beta^2w_{\max}^2}{1-\gamma_0} J  \right.\\
	&\qquad \left.+ \frac{\kappa^2}{1-\rho^2}\|x_0\|^2+2\beta^2 w_{\max}^2J	\right)\\
	& \quad +  L_{c1} \frac{\kappa}{1-\rho}(L_{\pi}+1)(\beta w_{\max}(L_{\pi}+1)+L_{c2}) \left( \frac{1}{1-\sqrt{\gamma_0/2}}\frac{\sqrt{\frac{\gamma_1}{2}}^{\mathbb M+1}-1}{\sqrt{\frac{\gamma_1}{2}}-1} (\|x_0\|+ \frac{\alpha_8}{\sqrt{\frac{\gamma_1}{2}}-1}) \right.\\
	& \qquad	\left.  +\frac{\beta w_{\max}}{1-\sqrt{\frac{\gamma_1}{2}}} J+ \frac{\kappa}{1-\rho}\|x_0\|+\beta w_{\max}J \right) \\
	& \leq \textup{AuxRegret} +  \alpha_6 (L_f(1+L_{\pi})\kappa)^{2\mathbb M} \textup{poly}(\|x_0\|, \beta w_{\max}, \bar \pi_0) + \alpha_{7} \beta^2 w_{\max}^2 J.
\end{align*}
where (a) uses $\tilde x_{t_j}=x_{t_j}$, Lemma \ref{lem: bdd on sum xtj}, and 	$\sum_{j=0}^{J-1}\|\hat x_{t_j}\|\leq \frac{\kappa}{1-\rho}\|x_0\|+\beta w_{\max}J$, 
$\sum_{j=0}^{J-1}\|\hat x_{t_j}\|^2\leq \frac{\kappa^2}{1-\rho^2}\|x_0\|^2+2\beta^2 w_{\max}^2J$ by the definition of $\hat x_t$. This completes the proof of Lemma \ref{lem: difference btw PoR and PsR}.

\section{Additional examples}
Consider system $\dot x=-x^3 +u$. Consider a point-wise min-norm policy
\begin{align*}
	\pi(x)=\begin{cases}
		x^3-2x & \text{if } x^2<2\\
		0 & \text{if } x^2\geq 2.
	\end{cases}
\end{align*} 
It is straigthtforward to verify that this policy satisfies Definitions \ref{def: exp ISS} and \ref{def: incremental global exp stable}.

\section{Numerical experiment details}
\label{appdx:numerical}

In this appendix we provide additional details for the experiments presented in \Cref{sec:experiments}.

Recall that the planar ``quadrotor'' has the state
$(\bx, \theta, \dot \bx, \dot \theta)$,
where $\bx \in \R^2$ denotes the two-dimensional position,
$\theta \in \R$ denotes the attitude angle,
and $\dot \bx \in \R^2,\ \dot \theta \in \R$ denote their respective velocities.
The control inputs $(u_1, u_2) \in \R^2$ are the thrusts of the left and right propellers.
The continuous-time dynamics are given by
\begin{equation}
\label{eq:pvtol-dynamics}
	m \ddot \bx = \begin{pmatrix}
		-(u_1 + u_2) \sin\theta \\
		-(u_1+u_2) \cos\theta - mg
	\end{pmatrix} - d_x \norm{\dot \bx} \dot \bx,
	\quad \quad
	\quad I \ddot \theta = r(u_1-u_2) - d_\theta |\dot \theta| \dot \theta,
\end{equation}
where $\norm{\cdot}$ denotes the Euclidean norm,
$g \approx 9.81$ is the gravitational constant,
$m > 0$ is the mass, $I > 0$ is the moment of inertia, $r > 0$ is the arm length,
and ${d_x > 0,\ d_\theta > 0}$ are drag constants.
For our experiment, we discretize the dynamics with symplectic Euler integration.
Our task is to fly the quadrotor towards the origin, which is expressed by the optimal control cost
$
	c_t(x_t, u_t) = \norm{\bx}_2^2.
$
We consider a parameterized family of controllers following a geometric proportional-derivative control law.
Our controllers are analogous to the widely-used 3D quadrotor controller proposed by \citet{lee2010geometric}.
Specifically, the control law computes a desired acceleration according to
\[
	\ddot \bx_\des = -k_p \bx - k_d \dot \bx,
\]
which implies a desired thrust direction
\[
	t_\des = \ddot \bx_\des + [0, \ g]^\top.
\]
This, in turn, implies a desired attitude angle $\theta_\des$ (computation omitted),
which is used to compute a desired angular acceleration
\[
	\ddot \theta_\des = -k_p^\theta (\theta - \theta_\des) - k_d^\theta \dot \theta.
\]
These are converted to a desired total thrust $h \in \R$ and torque $\tau \in \R$ using
\[
	h = \tilde m [\cos \theta, \ \sin \theta]^\top t_\des,
	\quad
	\tau = \tilde I \ddot \theta_\des,
\]
where $\tilde m$ and $\tilde I$ are estimates of the planar quadrotor's mass $m$ and moment of inertia $I$.
Finally, to avoid numerical instability for destabilizing controllers, we clamp the desired torque and thrust and generate the propeller inputs according by solving the linear system
\[
	\begin{pmatrix} \clip(h, 10^3) \\ \clip(\tau, 10^4) \end{pmatrix}
	= 
	\begin{pmatrix}
		1 & 1 \\
		r & -r \\
	\end{pmatrix}
	\begin{pmatrix} u_1 \\ u_2 \end{pmatrix}
\]
for $(u_1, u_2)$, where $\clip(x, a) = \max \{ \min \{ x, a \}, -a \}$.
\footnote{The larger limit for torque ensures that the closed-loop system can have faster attitude dynamics than position dynamics, which is necessary for quadrotor stabilization.
In real systems this is naturally achieved since typically $I/r \ll m$,
but our unrealistically large nominal moment of inertia $I/r = m$ necessitates large torques.}

The parameters of the controller are the gains $(k_p, k_d, k_p^\theta, k_d^\theta)$
and the system-identification estimates $(\tilde m,\ \tilde{I})$.
We select nominal gains $\bar k_p = 40,\ \bar k_d = 10,\ \bar k_p^\theta = 400,\ \bar k_d^\theta = 100$,
define the scale factors $S = \{ \frac{1}{10}, 1, 10 \}$,
and generate a pool of candidate controllers according to the rules
\begin{align*}
	k_p &\in \{ s \bar k_p : s \in S \}
	,\ &
	(k_d / k_p) &\in \{ s (\bar k_d / \bar k_p) : s \in S\}, \\
	k_p^\theta &\in \{ s \bar k_p^\theta : s \in S \}
	,\ &
	(k_d^\theta / k_p^\theta) &\in \{ s (\bar k_d^\theta / \bar k_p^\theta) : s \in S \}.
\end{align*}
These rules ensure that the ratio between the proportional and derivative gains do not become overly different from nominal, which helps ensure that a reasonable amount of the controllers are stabilizing.
This yields a pool of $81$ controllers.
We also consider that all controllers have an incorrect estimate $\tilde m = 2m$ but a correct estimate $\tilde I = I$.
The true system has the parameters $m = 1,\ r = 1,\ I = 1$
as well as drag constants $d_x = 10^{-4},\ d_\theta = 10^{-8}.$
We discretize with a time interval of $0.01$ seconds
and subject the system to thrust and torque disturbances sampled i.i.d. from a zero-mean Gaussian distribution with $\sigma = 0.1$. For all algorithms, we use
the decay parameters $\kappa = 1.1,\ \rho = 0.995,\ \beta w_{\max} = 4.35$.




%
%


\end{document}